\documentclass[11pt,reqno]{amsart}
\usepackage{amsmath,amssymb,latexsym,esint,cite,mathrsfs}
\usepackage{verbatim,wasysym}
\usepackage[left=2cm,right=2cm,top=2.5cm,bottom=2.5cm]{geometry}
\usepackage{tikz,enumitem,graphicx, subfig, microtype, color}
\usepackage{epic,eepic}

\usepackage[colorlinks=true,urlcolor=blue, citecolor=red,linkcolor=blue,
linktocpage,pdfpagelabels, bookmarksnumbered,bookmarksopen]{hyperref}
\usepackage[hyperpageref]{backref}
\usepackage[english]{babel}

\numberwithin{equation}{section}

\newtheorem{thm}{Theorem}[section]
\newtheorem{lem}[thm]{Lemma}

\newtheorem{Prop}[thm]{Proposition}

\newcommand{\R}{\mathbb{R}}

\newcommand\cp{\mathcal{P}}
\newcommand\cq{\mathcal{Q}}

\def\sp {\quad}

\begin{document}
	\baselineskip=14pt
	
	\title[Critical Choquard equation]{ Construction of infinitely many solutions for a critical Choquard equation via local Poho\v{z}aev identities}
	
	\author[F. Gao]{Fashun Gao}
	\author[V.\ Moroz]{Vitaly Moroz}
	\author[M. Yang]{Minbo Yang$^*$}
	\author[S. Zhao]{Shunneng Zhao}
	
	\address{Fashun Gao, \newline\indent Department of Mathematics and Physics, Henan University of Urban Construction, \newline\indent
		Pingdingshan, Henan, 467044, People's Republic of China}
	\email{fsgao@zjnu.edu.cn}
	
	\address{Vitaly Moroz,
		\newline\indent Mathematics Department, Swansea University,
		\newline\indent Bay Campus, Fabian Way, Swansea SA1 8EN, Wales, United Kingdom}
	\email{v.moroz@swansea.ac.uk}

	\address{Minbo Yang, \newline\indent Department of Mathematics, Zhejiang Normal University, \newline\indent
		Jinhua 321004, People's Republic of China}
	\email{mbyang@zjnu.edu.cn}

	\address{Shunneng Zhao, \newline\indent Department of Mathematics, Zhejiang Normal University, \newline\indent
		Jinhua 321004, People's Republic of China}
	\email{snzhao@zjnu.edu.cn}

	
	\thanks{$^\dag$Fashun Gao is partially supported by NSFC (11901155); Vitaly Moroz and Minbo Yang are partially supported by the Royal Society IEC\textbackslash NSFC\textbackslash 191022.}
	\thanks{$^\ddag$Minbo Yang is the corresponding author who is partially supported by NSFC (11971436, 12011530199) and ZJNSF(LZ22A010001, LD19A010001).}

	\begin{abstract}
		In this paper, we study a class of the critical Choquard equations with axisymmetric potentials,
		$$
		-\Delta u+ V(|x'|,x'')u
		=\Big(|x|^{-4}\ast |u|^{2}\Big)u\hspace{4.14mm}\mbox{in}\hspace{1.14mm} \mathbb{R}^6,
		$$
		where $(x',x'')\in \mathbb{R}^2\times\mathbb{R}^{4}$, $V(|x'|, x'')$ is a bounded nonnegative function in $\mathbb{R}^{+}\times\mathbb{R}^{4}$, and $*$ stands for the standard convolution. The equation is critical in the sense of the Hardy-Littlewood-Sobolev inequality. By applying a finite dimensional reduction argument and developing novel local Poho\v{z}aev identities, we prove that if the function $r^2V(r,x'')$ has a topologically nontrivial critical point then the problem admits  infinitely many solutions with arbitrary large energies.	
	\end{abstract}

	\maketitle
	
	\begin{center}
	\begin{minipage}{8.5cm}
			\small
			\tableofcontents
		\end{minipage}
\end{center}
%

\section{Introduction and main results}
\ \ \ \ The nonlinear Choquard equation
\begin{equation}\label{Nonlocal.S1}
	-\Delta u+ V(x)u=\Big(|x|^{-\mu}\ast |u|^{q}\Big)|u|^{q-2}u,\hspace{4.14mm} \mbox{in}\ \mathbb{R}^N
\end{equation}
arises in various fields of mathematical physics, such as the description of the quantum theory of a polaron at rest \cite{P1} and the Hartree-Fock theory of one-component plasma which models of the electron trapped
in its own hole \cite{Lieb1}.
First mathematical results on equation \eqref{Nonlocal.S1} are due to Lieb \cite{Lieb1}, who proved the existence and uniqueness, up to translations, of the ground state for \eqref{Nonlocal.S1}) with $N=3$, $\mu= 1$, $q = 2$ and $V$ a positive constant. Lieb's results were improved and extended by Lions \cite{Ls,Lions-I} established, in particular, the existence of a sequence of radially symmetric solutions.  In the case when $V$ is a constant, Moroz and Van Schaftingen \cite{MS1} proved regularity, positivity and radial symmetry of the ground states for an optimal range of $q$, and described asymptotic decay of solutions at infinity.
Ackermann \cite{AC} proposed a new approach to prove the existence of infinitely many geometrically distinct weak solutions in the case of a periodic potential $V$. If the potential $V$ has the form of a deep potential well $\lambda a(x)+1$ where $a(x)$ is a nonnegative continuous function such that $\Omega =$ int $(a^{-1}(0))$ is a non-empty bounded open set with smooth boundary, in \cite{ANY} the authors studied the existence and multiplicity of multi-bump shaped solution. The existence of nodal solutions for the nonlocal Choquard equation has also attracted a lot of interest, see  for example \cite{GV, RV} , the sign-changing solutions is hard to obtain due to the nonlocal interaction. The existence and concentration behavior of semiclassical solutions for the singularly perturbed subcritical Choquard equation have been considered in \cite{ACTY, AGSY, DGY, MS3, WW}.  Among them, Wei and Winter \cite{ WW} constructed families of solutions by a using a Lyapunov-Schmidt type reduction. Chen \cite{Cgy} constructed multiple semiclassical solutions to the Choquard equation with an external potential by Lyapunov-Schmidt reduction argument. Luo et al. \cite{LPW} established the uniqueness of positive solutions that concentrating at the nondegenerate critical points of the potential by combining a local type of Poho\v{z}aev's identities and blow-up techniques.

To state the main results, we first recall
the Hardy-Littlewood-Sobolev inequality (see \cite[Theorem 4.3]{LL}) to clarify the meaning of ``critical" for the Choquard equation.

\begin{Prop}\label{pro1.1}
	Let $t,\,r>1$ and $0<\mu<N$ be such that $\frac{1}{t}+\frac{\mu}{N}+\frac{1}{r}=2$.
	Then there is a sharp constant $C(N,\mu,t)$ such that, for $f\in L^{t}(\R^N)$
	and $h\in L^{r}(\R^N)$,
	$$
	\left|\int_{\R^{N}}\int_{\R^{N}}\frac{f(x)h(y)}{|x-y|^{\mu}}dxdy\right|
	\leq C(N,\mu,t) |f|_{L^t(\R^N)}|h|_{L^r(\R^N)}.
	$$
If $t=r=2N/(2N-\mu)$, then
$$
C(t,N,\mu,r)=C(N,\mu)=\pi^{\frac{\mu}{2}}\frac{\Gamma(\frac{N}{2}-\frac{\mu}{2})}{\Gamma(N-\frac{\mu}{2})}\left\{\frac{\Gamma(\frac{N}{2})}{\Gamma(N)}\right\}^{-1+\frac{\mu}{N}}.
$$
In this case there is equality is achieved if and only if $f\equiv Ch$ and
$$
h(x)=A(\gamma^{2}+|x-a|^{2})^{-(2N-\mu)/2}
$$
for some $A\in \mathbb{C}$, $0\neq\gamma\in\mathbb{R}$ and $a\in \mathbb{R}^{N}$.
\end{Prop}

According to Proposition \ref{pro1.1}, the functional
$$
\int_{\R^{N}}\int_{\R^{N}}\frac{|u(x)|^{p}|v(y)|^{p}}{|x-y|^{\mu}}dxdy
$$
is well defined in $H^1(\R^N)\times H^1(\R^N)$ if
$\frac{2N-\mu}{N}\leq p\leq\frac{2N-\mu}{N-2}$.
Here, it is quite natural to call $\frac{2N-\mu}{N}$ the lower Hardy-Littlewood-Sobolev critical exponent and $2_{\mu}^{\ast}=\frac{2N-\mu}{N-2}$ the upper Hardy-Littlewood-Sobolev critical exponent.

For the upper critical case, by using the moving plane methods in integral developed in \cite{CLO1,CLO2} ,  Lei \cite{Lei}, Du and Yang \cite{DY}, Guo et al. \cite{Guo2019} classified independently the positive solutions of the critical Hartree equation
\begin{equation}\label{eq1.5}
	-\Delta u=\Big(|x|^{-\mu}\ast |u|^{2_{\mu}^{\ast}}\Big)|u|^{2_{\mu}^{\ast}-2}u
	\ \ \mbox{in}\ \R^N,
\end{equation}
and proved that every positive
solution of \eqref{eq1.5} must assume the form
\begin{equation}\label{REL}
	U_{z,\lambda}(x)=S^{\frac{(N-\mu)(2-N)}{4(N-\mu+2)}}C(N,\mu)^{\frac{2-N}{2(N-\mu+2)}}[N(N-2)]^{\frac{N-2}{4}}\Big(\frac{\lambda}{1+\lambda^{2}|x-z |^{2}}\Big)^{\frac{N-2}{2}}.
\end{equation}
In \cite{GY} the authors
considered the Br\'{e}zis-Nirenberg type problem and established the existence, multiplicity and nonexistence of solutions for the nonlinear Choquard equation in bounded domain. Ghimenti and Pagliardini  in \cite{GP} studied the existence of multiple solutions for the Choquard equation with slightly subcritical exponent. In \cite{AGSY}, by investigating the ground states of the critical Choquard equation with constant coefficients, the authors studied the semiclassical limit problem for the singularly perturbed Choquard equation in $\R^3$ and characterized the concentration behavior by variational methods. The planar case was considered in \cite{ACTY}, where the authors established the existence of a ground state for the limit problem with critical exponential growth 
and then they also studied the concentration around the global minimum set. Gao and Yang in \cite{ GY3} investigated the existence result for the strongly indefinite Choquard equation with upper critical exponent in the whole space. In \cite{GSYZ}, Gao et al. studied the critical Choquard equations
with subcritical perturbation and potential functions that might change sign and proved the existence of Mountain-Pass solution via a nonlocal version of the concentration-compactness principle.  In \cite{DGY}, Ding et al. established a global compactness lemma and proved the multiplicity of high energy semiclassical states for a class of critical Choquard equations without lower perturbation using the Ljusternik-Schnirelman theory methods. In \cite{XZ1}, the authors studied the existence of saddle type solutions. For recent progress on the study of the Choquard equation we may refer the readers to \cite{ANY,  CT, DGY1, GY, guide} for details.

The aim of this paper is to establish the existence of infinitely many solutions for the critical Choquard equation with an axisymmetric potential $V(x)$ of a special form,
\begin{equation}\label{CFL}
	-\Delta u+  V(|x'|,x'')u
	=\Big(|x|^{-4}\ast |u|^{2}\Big)u\hspace{4.14mm}\mbox{in}\hspace{1.14mm} \mathbb{R}^6,
\end{equation}
where $(x', x'')\in \R^{2}\times\R^{4}$, the potential $V\geq0$ is bounded, belongs to $C^{1}$ and $V\not\equiv0$. The exponent $p=2$ is critical since $N=6$ and $\mu=4$.  We assume that the function $r^{2}V(r, x'')$ has a stable topologically nontrivial critical point in the following sense, which was introduced in \cite{PWY}:

\begin{itemize}
	\item[$\textbf{(V)}$]  The function $r^{2}V (r, x'')$ has a critical point $(r_0, x_0'')$ such that $r_0 > 0$, $V (r_0, x_0'') > 0$, and $$deg(\nabla(r^{2}V (r, x'')),(r_0, x_0'')) \neq 0.$$
\end{itemize}
The main result of this paper is the following theorem.

\begin{thm}\label{EXS}
	If $V(|x'|, x'')$ satisfies assumption $(V)$, then problem \eqref{CFL} has infinitely many solutions in $H^1(\R^6)$ and the energy of the solutions diverges to $+\infty$.
\end{thm}


As far as we are aware, this is the first result of its kind for nonlocal Choquard equations with axisymmetric potentials. We will  construct solutions of \eqref{CFL} with large number of bubbles in Theorem \ref{EXS} by using a reduction argument together with the novel local Poho\v{z}aev's  identities. The idea is inspired by recent works \cite{CWY, DLY,GPY, WY1} for non-singularly perturbed elliptic problems. It is well understood that the proofs of the existence of solutions with large number of bubbles for critical elliptic equations by reduction arguments depends heavily on the nondegeneracy of the solutions of the critical Lane-Emden equation. It is well known that the following equation
\begin{equation}\label{lcritical}
	-\Delta u=u^{\frac{N+2}{N-2}},~~x\in\mathbb{R}^{N}.
\end{equation}
has a unique family of positive two-parameter solutions of the form
\begin{equation}\label{U0}
	U_{\xi,\lambda}(x):=[N(N-2)]^{\frac{N-2}{4}}\Big(\frac{\lambda}{1+\lambda^2|x-\xi|^{2}}\Big)^{\frac{N-2}{2}}.
\end{equation}
Furthermore, equation \eqref{lcritical} has
an $(N+1)$-dimensional manifold of solutions given by
$$
\mathcal{Z}=\left\{z_{\lambda,\xi}=[N(N-2)]^{\frac{N-2}{4}}\Big(\frac{\lambda}{+\lambda^2|x-\xi|^{2}}\Big)^{\frac{N-2}{2}},
\xi\in\mathbb{R}^{N}, t\in\mathbb{R}^{+}\right\}.
$$
For every $Z\in \mathcal{Z}$, it is said to be nondegenerate in the sense that the linearized equation
\begin{equation}\label{Linearized}
	-\Delta v=Z^{\frac{4}{N-2}}v
\end{equation}
in $D^{1,2}(\mathbb{R}^N)$ only admits solutions of the form
$$
\eta=aD_{t}Z+\mathbf{b}\cdot\nabla Z,
$$
where $a\in\mathbb{R},\mathbf{b}\in\mathbb{R}^{N}$. In \cite{WY1}, Wei and Yan developed 
the technique that allowed to study the prescribed scalar curvature problem on $\mathbb{S}^N$
$$
-\Delta_{\mathbb{S}^N} u +\frac{N(N-2)}{2}u= K(x)u^{\frac{n+2}{n-2}}\hspace{4.14mm}\mbox{on}\hspace{1.14mm}  \mathbb{S}^N.
$$
Assuming that $K(x)$ is positive and rotationally symmetric and has a local maximum point between the poles, by using the stero-graphic projection, the prescribed scalar curvature problem can be reduced into
$$
-\Delta u = K(x)u^{\frac{n+2}{n-2}}\hspace{4.14mm}\mbox{in}\hspace{1.14mm}  \R^N.
$$
The authors took the number of the bubbles of the solutions as parameter and proved the existence of infinitely many non-radial positive solutions whose energy can be made arbitrarily large. We may also turn to the works by Deng, Lin, Yan \cite{DLY}, Guo, Peng, Yan \cite{GPY} and Li, Wei, Xu \cite{LWX} for the existence and local uniqueness of multi-bump solutions.
For the critical Schr\"{o}dinger equation
$$
-\Delta u+V(|x|)u = u^{\frac{n+2}{n-2}}\hspace{4.14mm}\mbox{in}\hspace{1.14mm}  \R^N,
$$
Chen, Wei and Yan \cite{CWY}
applied the reduction argument to study the existence of infinitely many positive solutions in the case of a  radially symmetric potentials $V$.
In \cite{PWY}, Peng, Wang and Yan developed a new idea that allowed also to construct bubbling solutions concentrating at saddle points of some functions. They used the Poho\v{z}aev identities to find algebraic equations which determine the location of the bubbles. Results on the existence of infinitely many solutions for other elliptic problems can also be found in \cite{AW, AWZ, WWY1, WWY2, WY2, WY3} and the references therein.

Define
$$\aligned
H_{s}=\Big\{&u\in D^{1,2}(\mathbb{R}^6),u(x_{1},-x_{2},x'')=u(x_{1},x_{2},x''),\\
&\hspace{4mm}u(r\cos\theta,r\sin\theta,x'')=u\Big(r\cos(\theta+\frac{2j\pi}{m}),r\sin(\theta+\frac{2j\pi}{m}),x''\Big)\Big\}
\endaligned$$
and let
$$
z_{j}=\Big(\overline{r}\cos\frac{2(j-1)\pi}{m},\overline{r}\sin\frac{2(j-1)\pi}{m},\overline{x}''\Big), \ j=1,\cdot\cdot\cdot,m,
$$
where $\overline{x}''$ is a vector in $\mathbb{R}^{4}$. By the weak symmetry of $V(x)$, we have $V (z_j ) = V (\overline{r}, \overline{x}'')$, $j = 1,\cdots, m$.
We will use $U_{z_j,\lambda}$ (see \eqref{REL}) as an approximate solution. Let $\delta> 0$ be a small
constant, such that $r^{2}V (r, x'') > 0$ if $|(r, x'')-(r_0, x_0'' )|\leq 10\delta$. Let $\xi(x) = \xi(|x'|, x'')$ be a
smooth function satisfying $\xi= 1$ if $|(r, x'')-(r_0, x_0'' )|\leq\delta$, $\xi= 0$ if $|(r, x'')-(r_0, x_0'' )| \geq2\delta$,
and $0\leq \xi\leq 1$. Denote
$$
Z_{z_j,\lambda}(x)=\xi U_{z_j,\lambda}(x), \
Z_{\overline{r},\overline{x}'',\lambda}(x)=\sum_{j=1}^{m}Z_{z_j,\lambda}(x), \
Z_{\overline{r},\overline{x}'',\lambda}^{\ast}(x)=\sum_{j=1}^{m}U_{z_j,\lambda}(x),
$$
and
$$
Z_{j,1}=\frac{\partial Z_{z_j,\lambda}}{\partial \lambda}, Z_{j,2}=\frac{\partial Z_{z_j,\lambda}}{\partial \overline{r}},
Z_{j,k}=\frac{\partial Z_{z_j,\lambda}}{\partial \overline{x}_{k}''},
\ \mbox{for }k=3,\cdot\cdot\cdot,N, \ j=1,\cdot\cdot\cdot,m.
$$

In this paper, we always assume that $m > 0$ is a large integer, $\lambda\in[L_{0}m^{2},L_{1}m^{2}]$
for some constants $L_1 > L_0 > 0$ and
$$
|(\overline{r}, \overline{x}'')-(r_0, x_0'')|\leq\vartheta<\delta,
$$
where $\vartheta> 0$ is a small constant.

In order to prove Theorem \ref{EXS}, we will prove the following result.
\begin{thm}\label{EXS1}
	Under the assumptions of Theorem \ref{EXS}, there exists a positive integer $m_0>0$, such that for any integer $m\geq m_0$, \eqref{CFL} has a solution $u_m$ of the form
	$$
	u_m=Z_{\overline{r}_{m},\overline{x}_{m}'',\lambda_{m}}+
	\phi_{\overline{r}_{m},\overline{x}_{m}'',\lambda_{m}}=\sum_{j=1}^{m}\xi U_{z_j,\lambda_{m}}+\phi_{\overline{r}_{m},\overline{x}_{m}'',\lambda_{m}},
	$$
	where $\phi_{\overline{r}_{m},\overline{x}_{m}'',\lambda_{m}}\in H_{s}$ and $\lambda_{m}\in\Big[L_0m^{2},L_1m^{2}\Big]$. Moreover, as $m\rightarrow\infty$, $(\overline{r}_m, \overline{x}_m'')\rightarrow(r_0, x_0'')$, and $\lambda_{m}^{-2}\|\phi_{\overline{r}_{m},\overline{x}_{m}'',\lambda_{m}}\|_{L^{\infty}}\rightarrow0$.
\end{thm}

Before we finish this section, we would like to point out the main difficulties in the study of nonlocal Choquard equations.  First of all, the uniqueness and nondegeneracy of the ground--states is largely open. The subcritical case for \eqref{Nonlocal.S1} was partially solved in \cite{Len, WW, X}.   In \cite{Len, WW} the authors proved the uniqueness and nondegeneracy for the case $N=3$, $\mu=1$ and $q=2$.  Chen \cite{Cgy} generalized the nondegeneracy result to higher dimensions and Xiang \cite{X} generalized the results further by showing the nondegeneracy when $q>2$ is close $2$. For the critical case \eqref{eq1.5}, the nondegeneracy property is mostly open. As far as we know, the only result is \cite{DY}, there the authors established the nondegeneracy result when $\mu$ is close to $N$.  Recently, the last two authors of the present paper obtained the nondegeneracy result for $N=6$, $\mu=4$. In addition, the nonlocal convolution part in \eqref{Nonlocal.S1} makes it much more difficult to establish the local Poho\v{z}aev identities and to obtain the error estimates.


The paper is organized as follows: In Section 2, we give some preliminary results including a nondegeneracy result for the critical Hartree equation \eqref{6}. In Section 3, we carry out the reduction procedure for the critical Choquard equation \eqref{CFL}. In Section 4, we prove our main results by establishing local Poho\v{z}aev identities.

\section{Preliminary results}
In this section, we will introduce some estimates involving the convolution term and introduce a nondegeneracy result for the limit critical Hartree equation \eqref{6}.

\begin{lem}\label{B2} (Lemma B.1, \cite{WY1}) For each fixed $k$ and $j$, $k\neq j$,  let
	$$
	g_{k,j}(x)=\frac{1}{(1+|x-z_{j}|)^{\alpha}}\frac{1}{(1+|x-z_{k}|)^{\beta}},
	$$
	where $\alpha\geq 1$ and $\beta\geq1$ are two constants.
Then, for any constants $0<\delta\leq\min\{\alpha,\beta\}$, there is a constant $C>0$, such that
	$$
	g_{k,j}(x)\leq\frac{C}{|z_{k}-z_{j}|^{\delta}}\Big(\frac{1}{(1+|x-z_{j}|)^{\alpha+\beta-\delta}}+\frac{1}{(1+|x-z_{k}|)^{\alpha+\beta-\delta}}\Big).
	$$
\end{lem}

\begin{lem}\label{B3} (Lemma B.2, \cite{WY1}) For any constant $0<\delta<N-2$, $N\geq5$, there is a constant $C>0$, such that
	$$
	\int_{\mathbb{R}^{N}}\frac{1}{|x-y|^{N-2}}\frac{1}{(1+|y|)^{2+\delta}}dy\leq \frac{C}{(1+|x|)^{\delta}}.
	$$
\end{lem}

Using the methods to prove Lemma B.2 in \cite{WY1}, we can also establish the following lemma.
\begin{lem}\label{B4}
	For $N=6$ and $1\leq i\leq m$, there is a constant $C>0$, such that
	$$
	|x|^{-4}\ast \frac{\lambda^{2}}{(1+\lambda|x-z_{i}|)^{6+\eta}}\leq \frac{C}{(1+\lambda|x-z_{i}|)^{4}},
	$$
where $\eta>0$.
\end{lem}
\begin{proof}
First, we notice that
	$$
	|x|^{-4}\ast \frac{\lambda^{2}}{(1+\lambda|x-z_{i}|)^{6+\eta}}=
	\int_{\mathbb{R}^{6}}\frac{1}{|y|^{4}}\frac{\lambda^{2}}{(1+\lambda|x-z_{i}-y|)^{6+\eta}}dy
	=
	\int_{\mathbb{R}^{6}}\frac{1}{|y|^{4}}\frac{1}{(1+|\lambda x-\lambda z_{i}-y|)^{6+\eta}}dy.
	$$
	
	Let $d=\frac{\lambda}{2}|x-z_{i}|>1$. Then, we have
	$$
	\int_{B_{d}(0)}\frac{1}{|y|^{4}}\frac{1}{(1+|\lambda x-\lambda z_{i}-y|)^{6+\eta}}dy\leq \frac{C}{(1+ d)^{6+\eta}}\int_{B_{d}(0)}\frac{1}{|y|^{4}}dy\leq \frac{C d^{2}}{(1+ d)^{6+\eta}}\leq\frac{ C}{(1+ d)^{4+\eta}},
	$$
	and
	$$\aligned
	\int_{B_{d}(\lambda x-\lambda z_{i})}\frac{1}{|y|^{4}}\frac{1}{(1+|\lambda x-\lambda z_{i}-y|)^{6+\eta}}dy
	\leq \frac{1}{d^{4}}\int_{B_{d}(0)}\frac{1}{(1+|y|)^{6+\eta}}dy
\leq \frac{C}{(1+ d)^{4}}.
	\endaligned$$

	Suppose that $y\in\mathbb{R}^{6}\backslash (B_{d}(0)\cup B_{d}(\lambda x-\lambda z_{i}))$. Then
	$$
	|\lambda x-\lambda z_{i}-y|\geq\frac{1}{2}|\lambda x-\lambda z_{i}|, |y|\geq\frac{1}{2}|\lambda x-\lambda z_{i}|,
	$$
	and, we have
	$$
	\frac{1}{|y|^{4}}\frac{1}{(1+|\lambda x-\lambda z_{i}-y|)^{6+\eta}}\leq \frac{ C}{(1+ d)^{4}}\frac{1}{|y|^{4}}\frac{1}{(1+|\lambda x-\lambda z_{i}-y|)^{2+\eta}}.
	$$
	If $|y|\leq2|\lambda x-\lambda z_{i}|$, then
	$$
	\frac{1}{|y|^{4}}\frac{1}{(1+|\lambda x-\lambda z_{i}-y|)^{2+\eta}}\leq
	\frac{C}{|y|^{4}(1+|\lambda x-\lambda z_{i}|)^{2+\eta}}\leq
	\frac{C_{1}}{|y|^{4}(1+|y|)^{2+\eta}}.
	$$
	If $|y|\geq 2|\lambda x-\lambda z_{i}|$, then $|\lambda x-\lambda z_{i}-y|\geq|y|-|\lambda x-\lambda z_{i}|\geq\frac{1}{2}|y|$. As a result,
	$$
	\frac{1}{|y|^{4}}\frac{1}{(1+|\lambda x-\lambda z_{i}-y|)^{2+\eta}}\leq
	\frac{C}{|y|^{4}(1+|y|)^{2+\eta}}.
	$$
	Thus, we have
$$
\int_{\mathbb{R}^{N}\backslash (B_{d}(0)\cup B_{d}(\lambda x-\lambda z_{i}))}\frac{1}{|y|^{4}}\frac{1}{(1+|\lambda x-\lambda z_{i}-y|)^{6+\eta}}dy
\leq \frac{ C}{(1+ d)^{4}}\int_{\mathbb{R}^{N}}\frac{1}{|y|^{4}}\frac{1}{(1+|y|)^{2+\eta}}dy
\leq \frac{ C}{(1+ d)^{4}}.
$$
\end{proof}

Using \eqref{REL} and the identity (see (37) in \cite{DHQWF} for example)
\begin{equation}\label{eq3.18}
\int_{\R^N}\frac{1}{|x-y|^{2s}}\Big(\frac{1}{1+|y|^{2}}\Big)^{N-s}dy
=I(s)\Big(\frac{1}{1+|x|^{2}}\Big)^{s},\ \ 0 < s < \frac{N}{2},
\end{equation}
where
$$
I(s)=\frac{\pi^{\frac{N}{2}}\Gamma(\frac{N-2s}{2})}{\Gamma(N-s)}, \ \ \mbox{and }\Gamma(s)=\int_0^{+\infty} x^{s-1}e^{-x}\,dx, s>0,
$$
we have
$$
|x|^{-4}\ast |U_{z,\lambda}(x)|^{2}
=\int_{\R^N}\frac{U_{z,\lambda}^2(y)}{|x-y|^{4}}dy
=CU_{z,\lambda}(x),
$$
where $N\geq5$ and $C=I(2)S^{\frac{(N-\mu)(2-N)}{4(N-\mu+2)}}C(N,\mu)^{\frac{2-N}{2(N-\mu+2)}}[N(N-2)]^{\frac{N-2}{4}}$. So, we have the following lemma.
\begin{lem}\label{P0}
		For $N=6$ and $1\leq i\leq m$, there is a constant $C>0$, such that
	$$
	|x|^{-4}\ast |U_{z_{i},\lambda}(x)|^{2}
	=CU_{z_{i},\lambda}(x).
	$$
\end{lem}

\begin{lem}\label{P1}
For every $i\neq1$, $\alpha >3$, there is a constant $C>0$, such that
\begin{equation}\label{p1}
	\int_{\mathbb{R}^{6}}\frac{1}{(1+|y-\lambda z_{1}|^{2})^{\alpha}}
	\frac{1}{(1+|y-\lambda z_{i}|^{2})^{2}}dy
	=\frac{C}{(\lambda|z_{1}-z_{i}|)^{4}}.
\end{equation}
\end{lem}
\begin{proof}
First, by Lemma \ref{B2}, we have
$$\aligned
\int_{\mathbb{R}^{6}}&\frac{1}{(1+|y-\lambda z_{1}|^{2})^{\alpha}}
	\frac{1}{(1+|y-\lambda z_{i}|^{2})^{2}}dy\\
&\leq\frac{C}{(\lambda|z_{1}-z_{i}|)^{4}}\int_{\mathbb{R}^{6}}\Big(\frac{1}{(1+|y-\lambda z_{1}|)^{2\alpha}}
	+\frac{1}{(1+|y-\lambda z_{i}|)^{2\alpha}}\Big)dy
=\frac{C}{(\lambda|z_{1}-z_{i}|)^{4}}.
\endaligned$$

Next, we notice that
$$
\int_{\mathbb{R}^{6}}\frac{1}{(1+|y-\lambda z_{1}|^{2})^{\alpha}}
\frac{1}{(1+|y-\lambda z_{i}|^{2})^{2}}dy
=\int_{\mathbb{R}^{6}}\frac{1}{(1+|y|^{2})^{\alpha}}
\frac{1}{(1+|y+\lambda z_{1}-\lambda z_{i}|^{2})^{2}}dy.
$$
Let $d=\frac{\lambda}{2}|z_{1}-z_{i}|>1$. Then, we have

$$
\int_{B_{d}(0)}\frac{1}{(1+|y|^{2})^{\alpha}}
\frac{1}{(1+|y+\lambda z_{1}-\lambda z_{i}|^{2})^{2}}dy
\geq \frac{1}{(1+(3d)^{2})^{2}}\int_{B_{1}(0)}\frac{1}{(1+|y|^{2})^{\alpha}}dy
\geq \frac{C}{d^{4}},
$$
and
$$\aligned
\int_{B_{d}(\lambda z_{i}-\lambda z_{1})}&\frac{1}{(1+|y|^{2})^{\alpha}}
\frac{1}{(1+|y+\lambda x_{1}-\lambda z_{i}|^{2})^{2}}dy\\
&\geq \frac{1}{(1+(3d)^{2})^{\alpha}}\int_{B_{1}(\lambda z_{i}-\lambda z_{1})}\frac{1}{(1+|y+\lambda z_{1}-\lambda z_{i}|^{2})^{2}}dy
\geq \frac{C}{d^{2\alpha}}.
\endaligned$$
Let $z\in\mathbb{R}^{6}\backslash (B_{d}(0)\cup B_{d}(\lambda z_{i}-\lambda z_{1}))$ with $|z|=2d$ and $|z+\lambda z_{1}-\lambda z_{i}|=2d$. Then, we have
$$\aligned
\int_{\mathbb{R}^{6}\backslash (B_{d}(0)\cup B_{d}(\lambda z_{i}-\lambda z_{1}))}&\frac{1}{(1+|y|^{2})^{\alpha}}
\frac{1}{(1+|y+\lambda z_{1}-\lambda z_{i}|^{2})^{2}}dy\\
&\geq\int_{B_{d}(z)}\frac{1}{(1+|y|^{2})^{\alpha}}
\frac{1}{(1+|y+\lambda z_{1}-\lambda z_{i}|^{2})^{2}}dy\\
&\geq \frac{ C}{(1+ (3d)^{2})^{2}}\int_{B_{1}(z)}\frac{1}{(1+|y|^{2})^{\alpha}}
dy
= \frac{ C}{d^{4}}.
\endaligned$$
Thus, we can obtain
$$
\int_{\mathbb{R}^{6}}\frac{1}{(1+|y-\lambda z_{1}|^{2})^{\alpha}}
	\frac{1}{(1+|y-\lambda z_{i}|^{2})^{2}}dy
\geq\frac{C}{(\lambda|z_{1}-z_{i}|)^{4}}.
$$
\end{proof}

To prove Theorem \ref{EXS}, we will use a nondegeneracy result. From \cite{Lei, DY}, we know
	\begin{equation}\label{F}
	U_{0,1}(x):=\sqrt{2}S^{-\frac{1}{2}}C(6,4)^{-\frac{1}{2}}24(\frac{1}{1+|x|^2})^{2},
\end{equation}
where $C(6,4)=\frac{1}{4\pi}\frac{\Gamma(2)}{\Gamma(4)}\lbrace \frac{
	\Gamma(6)}{\Gamma(3)}\rbrace^{\frac{1}{3}}$, is the unique family positive solutions of the critical Hartree equations:
\begin{equation}\label{6}
	-\Delta u=\frac{1}{2}(I_{2}\ast u^{2})u,\ \ \mbox{in} \ \ \R^6,
\end{equation}
where $I_{2}:\R^6\rightarrow \R$ is the Riesz potential defined at each point $x\in \R^N\backslash \lbrace 0\rbrace$ by
$$
I_{2}(x)=\frac{A_{2}}{|x|^{4}},\ \ \mbox{where } A_{2}=\frac{\Gamma(2)}{\Gamma(1)\pi^{\frac{N}{2}}2^{2}}.	
$$
Here the coefficient $\frac{1}{2}$ is taken to simplify the computations.
In \cite{YZ1}, Yang and Zhao proved a nondegeneracy result, i.e.
\begin{lem}\label{nondegeneracy}(see \cite[Theorem 2.1]{YZ1})
	The linearization of equation \eqref{6} around the solution $U_{0,1}$,
	\begin{equation}
		-\Delta \varphi=\frac{1}{2}(I_{2}\ast U_{0,1}^2)\varphi+(I_{2}\ast U_{0,1}\varphi)U_{0,1}
	\end{equation}
	only admits solutions in $D^{1,2}(\R^6)$ of the form
	$$
	\varphi=\overline{a}D_{\lambda}U_{0,1}+\mathbf{b}\cdot\nabla U_{0,1},\ \ \overline{a}\in \R,\ \ \mathbf{b}\in \R^6 .	
	$$
\end{lem}
	\begin{proof}
	We sketch the proof for the convenience of the readers.
		Equation \eqref{6} is equivalent to the following system
		\begin{equation}\label{3}
			\left\lbrace
			\begin{aligned}
				-\Delta u&=uw \sp ~in\sp \R ^6,\\
				-\Delta w&=\frac{1}{2}u^{2}\sp in\sp \R ^6.
			\end{aligned}
			\right.
		\end{equation}
We define the Talenti bubbles functions
		\begin{equation}\label{18}
			U_{0,1}(x):=\sqrt{2}S^{-\frac{1}{2}}C(6,4)^{-\frac{1}{2}}24(\frac{1}{1+|x|^2})^{2}~~~\text{and}	~~~W:=\frac{1}{8\omega_{6}}\int_{\R^6}\frac{U_{0,1}^2(y)}{|x-y|^4}dy.
		\end{equation}
		Then $(U_{0,1},W)$ solves system \eqref{3}.
		The linearized equation of system \eqref{3} around the solution $(U_{0,1},W)$ is given by
		\begin{equation}\label{5}
			\left\lbrace
			\begin{aligned}
				-\Delta \varphi&=\varphi W+\psi U_{0,1} \sp in\sp \R ^6,\\
				-\Delta \psi&=\varphi U_{0,1} \sp ~~~~~~~~~in\sp \R ^6.
			\end{aligned}
			\right.
		\end{equation}
		We aim to prove that equation \eqref{5}	only admits solutions of the form
		\begin{equation}\nonumber
			(\varphi, \psi)=\bar{a}(2U_{0,1}+rU_{0,1}^{'}, 2W+rW^{'})+\sum_{j=1}^{6}b{j}\frac{\partial(U_{0,1},W)}{\partial x_j},	
		\end{equation}
		where $\bar{a}, b_{j}\in \R$.

		For general $N$, let $r=|x|$ and $\theta=\frac{x}{|x|}\in {\mathbb{S}^{N-1}}$, we denote $\Delta_{r}$ respectively $\Delta_{\mathbb{S}^{N-1}}$ to be the Laplace operator in radial coordinates and the Laplace-Beltrami operator:
		\begin{equation}\nonumber
			\begin{aligned}
				\Delta_{r}=\frac{\partial^2}{\partial r^2}+\frac{N-1}{r}\frac{\partial}{\partial r},\\
				\Delta_{\mathbb{S}^{N-1}}=\frac{1}{\sqrt{g}}	\sum \frac{\partial}{\partial y_{j}}(\sqrt{g}g^{ij}\frac{\partial}{\partial y_{j}}),
			\end{aligned}
		\end{equation}
		where $g=\det{g^{ij}}$ and $[g^{ij}]=[g_{ij}]^{-1}$. Clearly
		\begin{equation}\label{7}
			\Delta=\Delta_{r}+\frac{1}{r^2}\Delta_{\mathbb{S}^{N-1}}	.\end{equation}
		Let
		\begin{equation}\nonumber
			\begin{aligned}
				\varphi_{k}(x)=\int_{\mathbb{S}^{5}}\varphi(r,\theta)Y_{k}(\theta) d\theta,\\
				\psi_{k}(x)=\int_{\mathbb{S}^{5}}\psi(r,\theta)Y_{k}(\theta) d\theta.
			\end{aligned}	
		\end{equation}
		Then, we can find
		\begin{equation}\label{9}
			\begin{aligned}
				\Delta \varphi_{k}&=\int_{\mathbb{S}^{5}}\Delta_{r}\varphi(r,\theta)Y_{k}(\theta) d\theta\\&=\int_{\mathbb{S}^{5}}\Delta \varphi Y_{k}d\theta-\frac{1}{r^2}\int_{\mathbb{S}^{5}}\Delta_{\mathbb{S}^{N-1}}\varphi Y_{k}d\theta\\&
				=-(\varphi_{k}W+\psi_{k}U_{0,1})+\frac{\lambda_{k}}{r^2}\varphi_{k},
			\end{aligned}
		\end{equation}
		and
		\begin{equation}\label{10}
			\Delta \psi_{k}=-\varphi_{k}U_{0,1}+\frac{\lambda_{k}}{r^2}\psi_{k}	.
		\end{equation}
		For any radial function $f(r)$, it is easy to check that
		\begin{equation}\nonumber
			(\Delta f)^{'}=(f^{''}(r)+\frac{N-1}{r}f^{'}(r))^{'}=f^{'''}(r)+\frac{N-1}{r}f^{''}(r)-\frac{N-1}{r^2}f^{'}(r).	
		\end{equation}
		Since $(U_{0,1},W)$ satisfies system \eqref{3} and is radially symmetric, then we have
		\begin{equation}\nonumber
			-\Delta U_{0,1}=-(U_{0,1}^{''}+\frac{5}{r}U_{0,1}^{'})=U_{0,1}W.
		\end{equation}
		thus
		\begin{equation}\nonumber
			U_{0,1}^{''}+\frac{5}{r}U_{0,1}^{'}+U_{0,1}W=0.
		\end{equation}
		Differentiating both sides with respect to $r$, we can find
		\begin{equation}\nonumber
		U_{0,1}^{'''}+\frac{5}{r}U_{0,1}^{''}-\frac{5}{r^2}U_{0,1}^{'}+U_{0,1}^{'}W+U_{0,1}W^{'}=0,
		\end{equation}
		equivalently
		\begin{equation}\label{11}
			\Delta U_{0,1}^{'}-\frac{5}{r^2}U_{0,1}^{'}+U_{0,1}^{'}W+U_{0,1}W^{'}=0,
		\end{equation}
		and
		\begin{equation}\label{12}
			\Delta W^{'}-\frac{5}{r^2}W^{'}+U_{0,1}U_{0,1}^{'}=0.
		\end{equation}
		
		If $k\geqslant 8$, then we claim that $\varphi_{k}\equiv\psi_{k}\equiv 0$. Suppose that the claim is false. Multiplying \eqref{9} and \eqref{10} by $U^{'}$ and $W^{'}$, respectively,
		and integrating over $B_{r}$ centered at the origin with radius $r$, we obtain
		\begin{equation}\label{13}
		0=\int_{B_{r}}\Delta \varphi_{k}U_{0,1}^{'}+\int_{B_{r}}\varphi_{k}WU_{0,1}^{'}+\int_{B_{r}}\psi_{k}U_{0,1}U_{0,1}^{'}-\int_{B_{r}}\frac{\lambda_{k}}{r^2}\varphi_{k}U_{0,1}^{'}.
		\end{equation}
		We can calculate the first term of the above equation as follows
		\begin{equation}\label{14}
		\begin{aligned}
		\int_{B_{r}}\Delta \varphi_{k}U_{0,1}^{'}&=\int_{B_{r}}div(\nabla \varphi_{k})U_{0,1}^{'}
		\\&
		=\int_{\partial B_{r}}U_{0,1}^{'}\frac{\partial \varphi_{k}}{\partial \nu}-\int_{\partial B_{r}}\varphi_{k}\frac{\partial U_{0,1}^{'}}{\partial \nu}+\int_{B_{r}}\Delta U_{0,1}^{'}\varphi_{k}.
		\end{aligned}	
		\end{equation}
		Substituting \eqref{11} and \eqref{14} into \eqref{13}, we have
		\begin{equation}\label{15}
		\begin{aligned}
		0&=\int_{B_{r}}\Delta \varphi_{k}U_{0,1}^{'}+\int_{B_{r}}\varphi_{k}WU_{0,1}^{'}+\int_{B_{r}}\psi_{k}U_{0,1}U_{0,1}^{'}-\int_{B_{r}}\frac{\lambda_{k}}{r^2}\varphi_{k}U_{0,1}^{'}\\ &
		=\int_{\partial B_{r}}(U_{0,1}^{'}\frac{\partial \varphi_{k}}{\partial \nu}-\varphi_{k}\frac{\partial U_{0,1}^{'}}{\partial \nu})-\int_{B_{r}}U_{0,1}W^{'}\varphi_{k}+\int_{B_{r}}\psi_{k}U_{0,1}U_{0,1}^{'}+\int_{B_{r}}\frac{5-\lambda_{k}}{r^2}\varphi_{k}U_{0,1}^{'}.
		\end{aligned}
		\end{equation}
		With similar arguments applied to \eqref{10}, we can see that
		\begin{equation}\label{16}
		\begin{aligned}
		0&=\int_{B_{r}}\Delta \psi_{k}W^{'}+\int_{B_{r}}\varphi_{k}U_{0,1}W^{'}-\int_{B_{r}}\frac{\lambda_{k}}{r^2}\psi_{k}U_{0,1}^{'}\\ &=\int_{\partial B_{r}}(W^{'}\frac{\partial \psi_{k}}{\partial \nu}-\psi_{k}\frac{\partial W^{'}}{\partial \nu})+\int_{B_{r}}U_{0,1}U_{0,1}^{'}\varphi_{k}-\int_{B_{r}}\psi_{k}U_{0,1}U_{0,1}^{'}+\int_{B_{r}}\frac{5-\lambda_{k}}{r^2}\psi_{k}W^{'}.	
		\end{aligned}
		\end{equation}	
		Adding \eqref{15} and \eqref{16},
		we can deduce
		\begin{equation}\label{17}
			\begin{aligned}
				0&=\underbrace{\int_{\partial B_{r}}(U_{0,1}^{'}\frac{\partial \varphi_{k}}{\partial \nu}-\varphi_{k}\frac{\partial U_{0,1}^{'}}{\partial \nu})}\limits_{:=J_{1}(r)}+\underbrace{\int_{\partial B_{r}}(W^{'}\frac{\partial \psi_{k}}{\partial \nu}-\psi_{k}\frac{\partial W^{'}}{\partial \nu})}\limits_{:=J_{2}(r)}+\underbrace{\int_{B_{r}}\frac{5-\lambda_{k}}{r^2}(\varphi_{k}U_{0,1}^{'}+\varphi_{k}U_{0,1}^{'})}\limits_{:=J_{3}(r)}.
			\end{aligned}	
		\end{equation}		
		By choosing suitable $r$ and we will estimate $J_{1}(r)$, $J_{2}(r)$ and $J_{3}(r)$, respectively.
		Since $\varphi_{k}\not\equiv 0$, without loss of generality, we may assume that there exists some $r_{1}>0$ such that $\varphi_{k}(r)<0$ in $(0,r_{1}), \varphi_{k}(r_{1})=0$. Then we can deduce that $\varphi^{'}_{k}(r_{1})>0$.
		
		Next, we are going to prove that $\psi_{k}(r)<0$ for $r$ small enough. Assuming the contrary, we deduce from \eqref{10} that $\psi_{k}(r)$ can not have a positive local maximum in $(0,r_{1})$, hence $\psi_{k}(r)>\psi_{k}(0)\geqslant 0$ for $r\in (0,r_{1})$.
		Combining \eqref{9} with \eqref{3}, we obtain
		\begin{equation}\nonumber
			\begin{aligned}
				0>\int_{B_{r_{1}}}(\frac{\lambda_{k}}{r^2}\varphi_{k}-\psi_{k}U_{0,1})U_{0,1}&=\int_{B_{r_{1}}}(\Delta \varphi_{k}+\varphi_{k}W)U_{0,1}\\&
=\int_{\partial B_{r_{1}}}U_{0,1}\frac{\partial \varphi_{k}}{\partial \nu}
				>0,
			\end{aligned}
		\end{equation}
		which leads to a contradiction. It is sufficient to show that there exists some $r_{2}>0$ such that $\psi_{k}(r)<0$ in $(0,r_{2})$, $\psi_{k}(r_{2})=0$ and $\psi^{'}_{k}(r_{2})>0$. We divide our argument into three cases:
		
		$\mathbf{(1).}$ If $r_{1}=r_{2}$, we take $r=r_{1}=r_{2}$. It is easy to check that $J_{1}(r)<0,J_{2}(r)<0, J_{3}(r)<0$, that contradicts to \eqref{17}.
		
		$\mathbf{(2).}$ If $r_{1}>r_{2}$, obviously $J_{2}(r_{2})<0, J_{3}(r_{2})<0$. The difficulty here is to evaluate $J_{1}(r_{2})$.
		Define
		\begin{equation}\nonumber
			F(r)=r^{5}U^{'}\varphi^{'}_{k}-r^{5}U^{''}\varphi_{k}.
		\end{equation}
		For any $r\in (r_{2}, r_{1})$, from \eqref{9} and \eqref{11} we have
		\begin{equation}\nonumber
		\begin{aligned}
		F^{'}(r)&=r^{5}(\varphi^{''}_{k}+\frac{5}{r}\varphi^{'}_{k})U_{0,1}^{'}-r^{5}(U_{0,1}^{'''}+\frac{5}{r}U_{0,1}^{''})\varphi_{k}\\&
=(\lambda_{k}-5)r^{3}U^{'}\varphi_{k}+r^{5}U_{0,1}(W^{'}\varphi_{k}-U_{0,1}^{'}\psi_{k}).
		\end{aligned}	
		\end{equation}
		To obtain a contradiction, assume that $\psi_{k}>0$ does not hold for all $r\in (r_{2}, r_{1})$. Then $\psi_{k}$ must have a zero for some $r_{0}\in (r_{2},r_{1})$, such that $\psi_{k}(r)>0, r_{2}<r<r_{0}$ and $\psi^{'}_{k}(r_{0})<0$. For any $r_{2}<r<r_{0}$, $\psi_{k}(r)$ must have a local maximum. From the fact that
		\begin{equation}
		\Delta \psi_{k}=-\varphi_{k}U_{0,1}+\frac{\lambda_{k}}{r^2}\psi_{k}>0,\ \  \sp r_{2}<r<r_{0},
		\end{equation}
		we know $\psi_{k}>0$ for all $r_{2}<r<r_{1}$. Since $F^{'}(r)>0, r_{2}<r<r_{1}$, then $F(r_{2})<F(r_{1})<0$. 
		Then we can deduce that
		\begin{equation}\nonumber
			J_{1}(r_{2})=\int_{\partial B_{r_{2}}}r^{-5}F(r)<0,
		\end{equation}
		which contradicts with \eqref{17}.
		
		$\mathbf{(3).}$ If $r_{2}>r_{1}$. Similar to the argument of $(2)$, we need to evaluate $J_{2}(r_{1})$ .
		Define
		\begin{equation}\nonumber
		G(r)=r^{5}W^{'}\psi^{'}_{k}-r^{5}W^{''}\psi_{k}.
		\end{equation}
		We know that, for any $r_{1}<r<r_{2}$, \begin{equation}\nonumber
			G^{'}(r)=(\lambda_{k}-5)r^{3}W^{'}\psi_{k}-r^{5}U_{0,1}(W^{'}\varphi_{k}-U_{0,1}^{'}\psi_{k}).
			\end{equation}
			In order to prove $G^{'}(r)>0$ for $r_{1}<r<r_{2}$, we proceed to show that $\varphi_{k}>0, r\in (r_{1},r_{2})$. On the contrary, suppose that $\varphi_{k}>0$ does not hold for all $r_{1}<r<r_{2}$. Then $\varphi_{k}$ must have a zero for some $r_{3}\in (r_{1},r_{2})$, such that $\varphi_{k}(r)>0, r_{1}<r<r_{3}$ and $\varphi^{'}_{k}(r_{3})<0$. Therefore
			\begin{equation}\nonumber
			\begin{aligned}
			0<\int_{B_{r_{3}}\backslash B_{r_{1}}}(\frac{\lambda_{k}}{r^2}\varphi_{k}-\psi_{k}U_{0,1})U_{0,1}&=\int_{B_{r_{3}}\backslash B_{r_{1}}}(\Delta \varphi_{k}+\varphi_{k}W)U_{0,1}\\&
=\int_{\partial B_{r_{3}}}U_{0,1}\frac{\partial \varphi_{k}}{\partial \nu}-\int_{\partial B_{r_{1}}}U_{0,1}\frac{\partial \varphi_{k}}{\partial \nu}
			<0,
			\end{aligned}
			\end{equation}
			Then we can deduce that $J_{2}(r_{1})<0$ and \eqref{17} gives a contradiction.
			
			From above arguments, we can find $(\varphi_{k},\psi_{k})\equiv 0$ for $k\geqslant 8$. Therefore, the linearization of system \eqref{3} around the solution $(U_{0,1},W)$ has a kernel with at most 7 dimensions. On the other hand, when $k=1$ we check at once that $(\varphi_{1},\psi_{1})=(2U_{0,1}+rU_{0,1}^{'},2W+rW^{'})$ is a solution,  when $k=2,\cdots,7$, $\frac{\partial(U_{0,1},W)}{\partial x_j} (1\leqslant j\leqslant 6)$ are solutions. 
			Since $(2U_{0,1}+rU_{0,1}^{'},2W+rW^{'})$ and $\frac{\partial(U_{0,1},W)}{\partial x_j} (1\leqslant j\leqslant 6)$ are linearly independent, the family of solutions of \eqref{5} is given by $ (\varphi, \psi)=a(2U_{0,1}+rU_{0,1}^{'}, 2W+rW^{'})+\sum\limits_{j=1}^{6}b{j}\frac{\partial(U_{0,1},W)}{\partial x_j}$ for some $a, b_{j}\in \R$.
			
		\end{proof}

\section{Finite-dimensional reduction}
In this section, we carry out the finite-dimensional reduction argument in a weighted space introduced in \cite{WY1}. In this way, we can obtain a good estimate for the error term. Let
$$
\|u\|_{\ast}=\sup_{x\in\mathbb{R}^6}\Big(\sum_{j=1}^{m}
\frac{1}{(1+\lambda|x-z_{j}|)^{2+\tau}}\Big)^{-1}\lambda^{-2}|u(x)|,
$$
and
$$
\|h\|_{\ast\ast}=\sup_{x\in\mathbb{R}^6}\Big(\sum_{j=1}^{m}
\frac{1}{(1+\lambda|x-z_{j}|)^{4+\tau}}\Big)^{-1}\lambda^{-4}|h(x)|,
$$
where $\tau=\frac{1}{2}$.

Consider
\begin{equation}\label{c1}
	\left\{\begin{array}{l}
		\displaystyle -\Delta \phi+ V(r,x'')\phi
		-\Big(|x|^{-4}\ast |Z_{\overline{r},\overline{x}'',\lambda}|^{2}\Big)\phi-2\Big(|x|^{-4}\ast (Z_{\overline{r},\overline{x}'',\lambda}
		\phi)\Big)Z_{\overline{r},\overline{x}'',\lambda}
		\\
		\displaystyle \hspace{10.14mm}=h+
		\sum_{l=1}^{6}c_{l}\sum_{j=1}^{m}[\Big(|x|^{-4}\ast |Z_{z_j,\lambda}|^{2}\Big)Z_{j,l}+2\Big(|x|^{-4}\ast (Z_{z_j,\lambda}
		Z_{j,l})\Big)Z_{z_j,\lambda}]\hspace{4.14mm}\mbox{in}\hspace{1.14mm} \mathbb{R}^6,\\
		\displaystyle \phi\in H_{s}, \ \ \sum_{j=1}^{m}\int_{\mathbb{R}^{6}}\Big[\Big(|x|^{-4}\ast |Z_{z_j,\lambda}|^{2}\Big)Z_{j,l}\phi +2\Big(|x|^{-4}\ast |Z_{z_j,\lambda}Z_{j,l}|\Big)Z_{z_j,\lambda}\phi\Big] dx=0,~~l=1,2,\cdots,6,
	\end{array}
	\right.
\end{equation}
for some real numbers $c_{l}$.

\begin{lem}\label{C1}
	Suppose that $\phi_{m}$ solves \eqref{c1} for $h = h_m$. If $\|h_{m}\|_{\ast\ast}\to 0$ as $m\to \infty$, then $\|\phi_{m}\|_{\ast}\to 0$.
\end{lem}
\begin{proof}
	We argue by contradiction. Suppose that there exist $m\rightarrow+\infty$, $\overline{r}_m\rightarrow r_{0}$, $\overline{y}_{m}''\rightarrow y_0''$, $\lambda_{m}\in[L_{0}m^{2},L_{1}m^{2}]$ and $\phi_{m}$ solving \eqref{c1} for $h=h_{m}$, $\lambda=\lambda_{m}$, $\overline{r}=\overline{r}_{m}$, $\overline{y}''=\overline{y}_{m}''$,
	with $\|h_{m}\|_{\ast\ast}\rightarrow0$ and $\|\phi_{m}\|_{\ast}\geq c>0$. We may assume that $\|\phi_{m}\|_{\ast}=1$.
	
	By \eqref{c1}, we have
	\begin{equation}\label{c2}
		\aligned
	&	|\phi_m(x)|\\&\leq
C\int_{\mathbb{R}^6}\frac{1}{|y-x|^{4}}\Big(|y|^{-4}\ast (Z_{\overline{r},\overline{y}'',\lambda}
|\phi_m|)\Big)Z_{\overline{r},\overline{x}'',\lambda}(y)dy+	C	\int_{\mathbb{R}^6}\frac{1}{|y-x|^{4}}\Big(|y|^{-4}\ast |Z_{\overline{r},\overline{x}'',\lambda}|^{2}\Big)|\phi_m(y)| dy
		\\
	&+C\sum_{l=1}^{N}|c_{l}|\Bigg[\Big|\sum_{j=1}^{m}\int_{\mathbb{R}^6}\frac{1}{|y-x|^{4}}\Big(|y|^{-4}\ast (Z_{z_j,\lambda}
		Z_{j,l})\Big)Z_{z_j,\lambda}(y)dy\Big|+\Big|\sum_{j=1}^{m}\int_{\mathbb{R}^6}\frac{1}{|y-x|^{4}}\Big(|y|^{-4}\ast |Z_{z_j,\lambda}|^{2}\Big)Z_{j,l}(y)dy\Big|\Bigg]\\
			&\hspace{4mm}+C\int_{\mathbb{R}^6}\frac{1}{|y-x|^{4}}|h_m(y)|dy.
		\endaligned
	\end{equation}

To estimate the first term in the right side of  \eqref{c2}, from the fact that
	$$
	Z_{\overline{r},\overline{y}'',\lambda}\leq C\sum_{j=1}^{m}\frac{\lambda^{2}}{(1+\lambda|x-z_{j}|)^{4}},
	$$
applying Lemma \ref{B4}, we have
	\begin{equation}\label{c02}
		\aligned
		\int_{\mathbb{R}^6}&\frac{1}{|y-x|^{4}}\Big(|y|^{-4}\ast (Z_{\overline{r},\overline{x}'',\lambda}
	|\phi_m|)\Big)Z_{\overline{r},\overline{x}'',\lambda}(y)dy\\
		&\leq C\|\phi_m\|_{\ast}\|Z_{\overline{r},\overline{x}'',\lambda}\|_{\ast}\int_{\mathbb{R}^6}\frac{1}{|y-x|^{4}}
		\sum_{j=1}^{m}\frac{\lambda^{2}}{(1+\lambda|y-z_{j}|)^{4}}\sum_{j=1}^{m}
		\frac{\lambda^{2}}{(1+\lambda|y-z_{j}|)^{2+\tau}} dy\\
		&=C\|\phi_m\|_{\ast}\lambda^{2}\int_{\mathbb{R}^6}\frac{1}{|y-\lambda x|^{4}}
		\sum_{j=1}^{m}\frac{1}{(1+|y-\lambda z_{j}|)^{4}}\sum_{j=1}^{m}
		\frac{1}{(1+|y-\lambda z_{j}|)^{2+\tau}} dy.
		\endaligned
	\end{equation}
Now, we may define
	$$
	\Omega_{j}=\left\{x=(x',x'')\in\mathbb{R}^{2}\times\mathbb{R}^{4}:\Big\langle\frac{x'}{|x'|},
	\frac{z_{j}'}{|z_{j}'|}\Big\rangle\geq\cos\frac{\pi}{m}\right\}, \ j=1,\cdot\cdot\cdot,m.
	$$
Then for $y\in\Omega_{1}$, we have $|y-\lambda z_{j}|\geq|y-\lambda z_{1}|$. Using Lemma \ref{B2}, we obtain
$$\aligned
\sum_{j=2}^{m}\frac{1}{(1+|y-\lambda z_{j}|)^{4}}&\leq
\frac{1}{(1+|y-\lambda z_{1}|)^{2}}\sum_{j=2}^{m}\frac{1}{(1+|y-\lambda z_{j}|)^{2}}\\
&\leq\frac{C}{(1+|y-\lambda z_{1}|)^{4-\tau}}\sum_{j=2}^{m}\frac{1}{(\lambda|z_{1}-z_{j}|)^{\tau}}.
\endaligned$$
From the proof in Lemma B.3 in \cite{WY1}, we know
$$
\sum_{j=2}^{m}\frac{1}{(\lambda|z_{1}-z_{j}|)^{\tau}}\leq C,
$$
consequently, we have
\begin{equation}\label{5141}
\sum_{j=2}^{m}\frac{1}{(1+|y-\lambda z_{j}|)^{4}}\leq\frac{C}{(1+|y-\lambda z_{1}|)^{4-\tau}}.
\end{equation}
Similarly, we have
\begin{equation}\label{5142}
\sum_{j=2}^{m}\frac{1}{(1+|y-\lambda z_{j}|)^{2+\tau}}
\leq\frac{C}{(1+|y-\lambda z_{1}|)^{2}}.
\end{equation}
From \eqref{5141} and \eqref{5142}, for $y\in\Omega_{1}$, using Lemma \ref{B2} again, we know
\begin{equation}\label{c03}
	\aligned
	\sum_{j=1}^{m}\frac{1}{(1+|y-\lambda z_{j}|)^{4}}\sum_{j=1}^{m}
	\frac{1}{(1+|y-\lambda z_{j}|)^{2+\tau}}
	\leq\frac{C}{(1+|y-\lambda z_{1}|)^{6-\tau}}.
	\endaligned
\end{equation}
Therefore, by Lemma \ref{B3}, we have
$$\aligned
\int_{\Omega_{1}}\frac{1}{|y-\lambda x|^{4}}
\sum_{j=1}^{m}\frac{1}{(1+|y-\lambda z_{j}|)^{4}}\sum_{j=1}^{m}
\frac{1}{(1+|y-\lambda z_{j}|)^{2+\tau}} dy
\leq\frac{C}{(1+\lambda|x-z_{1}|)^{4-\tau}}.
\endaligned$$
In the following, we denote $\theta=2-2\tau$. Thus
$$\aligned
\int_{\mathbb{R}^6}\frac{1}{|y-\lambda x|^{4}}
\sum_{j=1}^{m}\frac{1}{(1+|y-\lambda z_{j}|)^{4}}\sum_{j=1}^{m}
\frac{1}{(1+|y-\lambda z_{j}|)^{2+\tau}} dy
\leq
\sum_{j=1}^{m}\frac{C}{(1+\lambda|x-z_{j}|)^{2+\tau+\theta}}.
\endaligned$$
Combining this with \eqref{c02}, we have
\begin{equation}\label{c00001}
	\int_{\mathbb{R}^6}\frac{1}{|y-x|^{4}}\Big(|y|^{-4}\ast (Z_{\overline{r},\overline{x}'',\lambda}
	|\phi_m|)\Big)Z_{\overline{r},\overline{x}'',\lambda}(y)dy
	\leq C \|\phi_m\|_{\ast}\lambda^{2}\sum_{j=1}^{m}
	\frac{1}{(1+\lambda|x-z_{j}|)^{2+\tau+\theta}}.
\end{equation}
Similarly, we can obtain the estimate for the second term in the right side of  \eqref{c2} as
\begin{equation}\label{5143}
\int_{\mathbb{R}^6}\frac{1}{|y-x|^{4}}\Big(|y|^{-4}\ast |Z_{\overline{r},\overline{x}'',\lambda}|^{2}\Big)|\phi_m(y)| dy\leq C \|\phi_m\|_{\ast}\lambda^{2}\sum_{j=1}^{m}
\frac{1}{(1+\lambda|x-z_{j}|)^{2+\tau+\theta}}.
\end{equation}

For the third term in the right side of  \eqref{c2}, if $l=1$, we can get
$$\aligned
||y|^{-4}\ast (Z_{z_j,\lambda}
Z_{j,1})|
\leq &C|y|^{-4}\ast \frac{\lambda^{3}}{(1+\lambda^{2}|y-z_{j}|^{2})^{5}}+C|y|^{-4}\ast \frac{\lambda^{3}}{(1+\lambda^{2}|y-z_{j}|^{2})^{4}}\\
\leq &C|y|^{-4}\ast \frac{\lambda^{3}}{(1+\lambda|y-z_{j}|)^{10}}+C|y|^{-4}\ast \frac{\lambda^{3}}{(1+\lambda|y-z_{j}|)^{8}}.
\endaligned$$
Thus, by Lemmas \ref{B4} and \ref{B3}, we have
$$
\left|\int_{\mathbb{R}^6}\frac{1}{|y-x|^{4}}\Big(|y|^{-4}\ast (Z_{z_j,\lambda}
Z_{j,1})\Big)Z_{z_j,\lambda}(y)dy\right|\leq C \lambda
\frac{1}{(1+\lambda|x-z_{j}|)^{2+\tau}}.
$$
While for $l\neq1$, we can obtain
$$
\left|\int_{\mathbb{R}^6}\frac{1}{|y-x|^{4}}\Big(|y|^{-4}\ast (Z_{z_j,\lambda}
Z_{j,l})\Big)Z_{z_j,\lambda}(y)dy\right|\leq C \lambda^{3}
\frac{1}{(1+\lambda|x-z_{j}|)^{2+\tau}},
$$
and so,
\begin{equation}\label{5144}
\left|\sum_{j=1}^{m}\int_{\mathbb{R}^6}\frac{1}{|y-x|^{4}}\Big(|y|^{-4}\ast (Z_{z_j,\lambda}
Z_{j,l})\Big)Z_{z_j,\lambda}(y)dy\right|\leq C \lambda^{2+n_{l}}\sum_{j=1}^{m}
\frac{1}{(1+\lambda|x-z_{j}|)^{2+\tau}},
\end{equation}
where $n_{1}=-1$, $n_{l}=1,l=2,\cdot\cdot\cdot,6$. Similarly, we have
\begin{equation}\label{5145}
\left|\sum_{j=1}^{m}\int_{\mathbb{R}^6}\frac{1}{|y-x|^{4}}\Big(|y|^{-4}\ast |Z_{z_j,\lambda}|^{2}\Big)Z_{j,l}(y)dy\right|
\leq C \lambda^{2+n_{l}}\sum_{j=1}^{m}
\frac{1}{(1+\lambda|x-z_{j}|)^{2+\tau}}.
\end{equation}
For the last term,  by Lemma \ref{B3}, we know
\begin{equation}\label{5146}
\int_{\mathbb{R}^6}\frac{1}{|y-x|^{4}}|h_m(y)|dy\leq C \|h_m\|_{\ast\ast}\lambda^{2}\sum_{j=1}^{m}
\frac{1}{(1+\lambda|x-z_{j}|)^{2+\tau}}.
\end{equation}

In the following, we are going to estimate $c_l$, $l = 1, 2, \cdot\cdot\cdot,6$. Multiplying \eqref{c1} by $Z_{1,t}$, $(t = 1, 2, \cdot\cdot\cdot,6)$ and integrating, we see that $c_l$ satisfies
\begin{equation}\label{c3}
	\aligned
\sum_{l=1}^{6}\sum_{j=1}^{m}	&\Big\langle \Big(|x|^{-4}\ast |Z_{z_j,\lambda}|^{2}\Big)Z_{j,l}+2\Big(|x|^{-4}\ast (Z_{z_j,\lambda}
	Z_{j,l})\Big)Z_{z_j,\lambda}, Z_{1,t}\Big\rangle c_{l}\\
	=&
	\Big\langle-\Delta \phi_m+ V(r,x'')\phi_m
	-\Big(|x|^{-4}\ast |Z_{\overline{r},\overline{x}'',\lambda}|^{2}\Big)\phi_m
	-2\Big(|x|^{-4}\ast (Z_{\overline{r},\overline{x}'',\lambda}
	\phi_m)\Big)Z_{\overline{r},\overline{x}'',\lambda}, Z_{1,t}\Big\rangle-\langle h_m, Z_{1,t}\rangle.
	\endaligned
\end{equation}

From Lemma 2.1 in \cite{PWY}, we know that
\begin{equation}\label{c4}
	|\langle h_m, Z_{1,t}\rangle|\leq C\lambda^{n_{t}}\|h_m\|_{\ast\ast},
\end{equation}
and
\begin{equation}\label{c5}
	|\langle V(r,x'')\phi_m, Z_{1,t}\rangle|\leq O(\frac{\lambda^{n_{t}}\|\phi_m\|_{\ast}}{\lambda^{1+\varepsilon}}),
\end{equation}
where $\varepsilon>0$ is small constant.
Since
$$\aligned
|x|^{-4}\ast \frac{\xi(x)\lambda^{2}}{(1+\lambda|x-z_{i}|)^{6+\tau}}=
&\int_{\mathbb{R}^{6}}\frac{1}{|y|^{4}}\frac{\xi(x-y)\lambda^{2}}{(1+\lambda|x-z_{i}-y|)^{6+\tau}}dy\\
\leq&\int_{B_{2\delta}(x-(r_0, x_0'' ))}\frac{\lambda^{2}}{(\lambda|x-z_{i}-y|)^{5}}dy
=O(\frac{1}{\lambda^{3}}),
\endaligned$$
and
$$\aligned
&|x|^{-4}\ast \Big(\frac{\xi(x)\lambda^{2}}{(1+\lambda|x-z_{i}|)^{4}}\frac{1}{(1+\lambda|x-z_{j}|)^{2+\tau}}\Big)\\
\leq&\frac{C}{(\lambda|z_{i}-z_{j}|)^{\frac{\tau}{2}}}\int_{\mathbb{R}^{6}}\frac{1}{|y|^{4}}
\Big(\frac{\xi(x-y)\lambda^{2}}{(1+\lambda|x-z_{i}-y|)^{6+\frac{\tau}{2}}}
+\frac{\xi(x-y)\lambda^{2}}{(1+\lambda|x-z_{j}-y|)^{6+\frac{\tau}{2}}}\Big)dy\\
=&O(\frac{1}{\lambda^{3}}),  \ \  j\neq i,
\endaligned$$
thanks to Lemma \ref{B2}, we have
$$\aligned
||x|^{-4}\ast (Z_{\overline{r},\overline{x}'',\lambda} \phi_{m})|\leq &C\|\phi_m\|_{\ast}\int_{\mathbb{R}^{6}}\frac{1}{|y|^{4}}\sum_{j=1}^{m}\frac{\xi(x-y)\lambda^{2}}{(1+\lambda|x-y-z_{j}|)^{4}}\sum_{j=1}^{m}
\frac{\lambda^{2}}{(1+\lambda|x-y-z_{j}|)^{2+\tau}}dy\\
=& O(\frac{m^{2}\|\phi_m\|_{\ast}}{\lambda}).
\endaligned$$
So,
$$\aligned
\int_{\mathbb{R}^6}&\Big(|x|^{-4}\ast (Z_{\overline{r},\overline{x}'',\lambda}
\phi_{m})\Big)Z_{\overline{r},\overline{y}'',\lambda} Z_{1,t}dx\\
&\leq C\|\phi_m\|_{\ast}\|Z_{\overline{r},\overline{x}'',\lambda}\|_{\ast}\frac{m^{2}}{\lambda}
\int_{\mathbb{R}^6}\sum_{j=1}^{m}
\frac{\lambda^{2}}{(1+\lambda|x-z_{j}|)^{2+\tau}}
\frac{\xi\lambda^{2+n_{t}}}{(1+\lambda|x-z_{1}|)^{4}}dx\\
&\leq C\|\phi_m\|_{\ast}\frac{m^{2}\lambda^{n_{t}}}{\lambda}
\frac{m}{\lambda^{2}}\int_{\mathbb{R}^6}
\frac{1}{(1+|x-\lambda z_{1}|)^{6+\tau}}dx\\
&=O(\frac{\lambda^{n_{t}}\|\phi_m\|_{\ast}}{\lambda^{1+\varepsilon}}),
\endaligned$$
for some small constant $\varepsilon>0$. 
Similarly, we also have
$$
\Big\langle\Big(|x|^{-4}\ast |Z_{\overline{r},\overline{x}'',\lambda}|^{2}\Big)\phi_m, Z_{1,t}\Big\rangle= O\Big(\frac{\lambda^{n_{t}}\|\phi_m\|_{\ast}}{\lambda^{1+\varepsilon}}\Big).
$$
Notice that
$$\aligned
\int_{\mathbb{R}^{6}}\Delta Z_{1,t}\phi_mdx
=\int_{\mathbb{R}^{6}}(\xi\Delta (U_{z_{1},\lambda})_{t}+(U_{z_{1},\lambda})_{t}\Delta\xi +2\nabla\xi\nabla (U_{z_{1},\lambda})_{t}) \phi_mdx,
\endaligned$$
where
$$
(U_{z_{1},\lambda})_{1}=\frac{\partial U_{z_{1},\lambda}}{\partial \lambda}, (U_{z_{1},\lambda})_{2}=\frac{\partial U_{z_{1},\lambda}}{\partial \overline{r}},
(U_{z_{1},\lambda})_{k}=\frac{\partial U_{z_{1},\lambda}}{\partial \overline{x}_{k}''},
\ \mbox{for }k=3,\cdot\cdot\cdot,N.
$$
By applying the argument as above, we have
$$\aligned
\int_{\mathbb{R}^{6}}\xi\Delta (U_{z_{1},\lambda})_{t} \phi_mdx
&\leq C\|\phi_m\|_{\ast}\sum_{j=1}^{m}\int_{\mathbb{R}^6}\int_{\mathbb{R}^6}
\frac{\xi\lambda^{4}}{(1+\lambda|x-z_{1}|)^{8}}
\frac{1}{|x-y|^{4}}\frac{\lambda^{2+n_{t}}}{(1+\lambda|y-z_{1}|)^{4}}
\frac{\lambda^{2}}{(1+\lambda|y-z_{j}|)^{2+\tau}}dxdy\\
&\leq C\frac{\lambda^{n_{t}}}{\lambda}\|\phi_m\|_{\ast}m\int_{\mathbb{R}^6}
\frac{\lambda^{2}}{(1+\lambda|y-z_{1}|)^{4}}\frac{\lambda^{2}}{(1+\lambda|y-z_{1}|)^{2+\tau}}dy\\
&=O(\frac{\lambda^{n_{t}}\|\phi_m\|_{\ast}}{\lambda^{\frac{5}{2}}}).
\endaligned$$
On the other hand, direct calculation gives
$$
\int_{\mathbb{R}^{6}}(U_{z_{1},\lambda})_{t}\Delta\xi \phi_mdx
\leq C\|\phi_m\|_{\ast}
\sum_{j=1}^{m}\int_{\mathbb{R}^6}
\frac{\lambda^{2}}{(1+\lambda|x-z_{j}|)^{2+\tau}}
\frac{|\Delta\xi|\lambda^{2+n_{t}}}{(1+\lambda|x-z_{1}|)^{4}}dx
= O(\frac{\lambda^{n_{t}}\|\phi_m\|_{\ast}}{\lambda^{2}}),
$$
and
$$
\int_{\mathbb{R}^{6}}\nabla\xi\nabla (U_{z_{1},\lambda})_{t} \phi_mdx
\leq C\|\phi_m\|_{\ast}
\sum_{j=1}^{m}\int_{\mathbb{R}^6}
\frac{|\nabla\xi|\lambda^{2}}{(1+\lambda|x-z_{j}|)^{2+\tau}}
\frac{\lambda^{3+n_{t}}}{(1+\lambda|x-z_{1}|)^{5}}dx
= O(\frac{\lambda^{n_{t}}\|\phi_m\|_{\ast}}{\lambda^{2}}).
$$
So,
$$
\langle-\Delta \phi_m, Z_{1,t}\rangle= O\Big(\frac{\lambda^{n_{t}}\|\phi_m\|_{\ast}}{\lambda^{2}}\Big),
$$
Consequently,
\begin{equation}\label{c6}
	\Big\langle-\Delta \phi_m
	-\Big(|x|^{-4}\ast |Z_{\overline{r},\overline{x}'',\lambda}|^{2}\Big)\phi_m
	-2\Big(|x|^{-4}\ast (Z_{\overline{r},\overline{x}'',\lambda}
	\phi)\Big)Z_{\overline{r},\overline{y}'',\lambda}, Z_{1,t}\Big\rangle= O(\frac{\lambda^{n_{t}}\|\phi_m\|_{\ast}}{\lambda^{1+\varepsilon}}).
\end{equation}

Combining \eqref{c4}-\eqref{c6}, we have
\begin{equation}\label{c7}
	\aligned
	\Big\langle-\Delta \phi_m&+ V(r,y'')\phi_m
	-\Big(|x|^{-4}\ast |Z_{\overline{r},\overline{x}'',\lambda}|^{2}\Big)\phi_m-2\Big(|x|^{-4}\ast (Z_{\overline{r},\overline{x}'',\lambda}
	\phi)\Big)Z_{\overline{r},\overline{x}'',\lambda}, Z_{1,t}\Big\rangle-\langle h_m, Z_{1,t}\rangle\\
	=& O\Big(\frac{\lambda^{n_{t}}\|\phi_m\|_{\ast}}{\lambda^{1+\varepsilon}}+\lambda^{n_{t}}\|h_m\|_{\ast\ast}\Big).
	\endaligned
\end{equation}
On the other hand, it is easy to check that
\begin{equation}\label{c8}
	\sum_{j=1}^{m}\Big\langle \Big(|x|^{-4}\ast |Z_{z_j,\lambda}|^{2}\Big)Z_{j,l}, Z_{1,t}\Big\rangle=(\overline{c}+o(1))\delta_{tl}\lambda^{n_{l}}\lambda^{n_{t}},
\end{equation}
and
\begin{equation}\label{c81}
	\sum_{j=1}^{m}\Big\langle \Big(|x|^{-4}\ast (Z_{z_j,\lambda}
	Z_{j,l})\Big)Z_{z_j,\lambda}, Z_{1,t}\Big\rangle=(\overline{c}'+o(1))\lambda^{n_{l}}\lambda^{n_{t}},
\end{equation}
for some constant $\overline{c} > 0$ and $\overline{c}' > 0$.
And so, from\eqref{c7}, \eqref{c8} and \eqref{c81} into \eqref{c3}, we know
\begin{equation}\label{c9}
	c_{l}=\frac{1}{\lambda^{n_{l}}}(o(\|\phi_m\|_{\ast})+O(\|h_m\|_{\ast\ast})).
\end{equation}
Thus,
\begin{equation}\label{c10}
	\|\phi_m\|_{\ast}\leq o(1)+\|h_m\|_{\ast\ast}+\frac{\sum_{j=1}^{6}\frac{1}{(1+\lambda|x-z_{j}|)^{2+\tau+\theta}}}
	{\sum_{j=1}^{6}\frac{1}{(1+\lambda|x-z_{j}|)^{2+\tau}}}.
\end{equation}
Since $\|\phi_m\|_{\ast}= 1$, we obtain from \eqref{c10} that there is $R > 0$ such that
\begin{equation}\label{c11}
	\|\lambda^{-2}\phi_m\|_{L^{\infty}(B_{\frac{R}{\lambda}}(z_{j}))}\geq a>0,
\end{equation}
for some $j$. However, let $\overline{\phi}_m (x)=\lambda^{-2}\phi_m(\lambda(x-z_{j}))$, then
$$
\int_{\R^N}(|\nabla \overline{\phi}_{m}|^{2}+ V(r,x'')|\overline{\phi}_{m}|^{2})dx\leq C.
$$
Thus there is a $v\in D^{1,2}(\mathbb{R}^{6})$, such that
$$
\overline{\phi}_m\rightharpoonup v,\ \ \mbox{weakly in }D^{1,2}(\mathbb{R}^{6})
$$
and
$$
\overline{\phi}_m\rightarrow v,\ \ \mbox{strongly in }L_{loc}^{2}(\mathbb{R}^{6}),
$$
 as $m\rightarrow+\infty$ . Therefore it follows that $v\in D^{1,2}(\mathbb{R}^{6})$ satisfies
\begin{equation}\label{ib5}
	-\Delta v
	=\Big(\int_{\mathbb{R}^{6}}\frac{|U_{0,\Lambda}(y)|^{2}}{|x-y|^{4}}dy\Big)v
	+2\Big(\int_{\mathbb{R}^{6}}\frac{U_{0,\Lambda}(y)v(y)}{|x-y|^{4}}dy\Big)U_{0,\Lambda}\hspace{4.14mm}\mbox{in}\hspace{1.14mm} \mathbb{R}^6,
\end{equation}
for some $\Lambda\in[\Lambda_{1}, \Lambda_{1}]$. Since $v$ is perpendicular to the kernel of \eqref{ib5}, by the non-degeneracy of $U_{0,1}$ in Theorem \ref{nondegeneracy},  we can conclude that $v = 0$, this contradicts with \eqref{c11}.
\end{proof}

Let
$$
E=\Big\{\phi\in H_{s}:\sum_{j=1}^{m}\int_{\mathbb{R}^{6}}\Big[\Big(|x|^{-4}\ast |Z_{z_j,\lambda}|^{2}\Big)Z_{j,l}\phi +2\Big(|x|^{-4}\ast |Z_{z_j,\lambda}Z_{j,l}|\Big)Z_{z_j,\lambda}\phi\Big] dx=0,~~l=1,2,\cdots,6\Big\}.
$$
endowed with the usual inner product $[\phi,\psi]=\int_{\mathbb{R}^{6}}\nabla\phi\nabla\psi dx$. Problem \eqref{c1} is equivalent to that of finding a $\phi\in E$ such that
$$
[\phi,\psi]=\langle -V(r,x'')\phi
+\Big(|x|^{-4}\ast |Z_{\overline{r},\overline{x}'',\lambda}|^{2}\Big)\phi+2\Big(|x|^{-4}\ast (Z_{\overline{r},\overline{x}'',\lambda}
\phi)\Big)Z_{\overline{r},\overline{x}'',\lambda}+h,\psi\rangle, \ \ \forall\psi\in E.
$$
Similar to the proof of Proposition 4.1 in \cite{DFM},  Riesz's representation theorem and Fredholm's alternative Theorem guarantee the existence of unique solution for any
$h$ provided the following equation
\begin{equation}\label{0}
	\left\{\begin{array}{l}
		\displaystyle -\Delta \phi+ V(r,x'')\phi
		-\Big(|x|^{-4}\ast |Z_{\overline{r},\overline{x}'',\lambda}|^{2}\Big)\phi-2\Big(|x|^{-4}\ast (Z_{\overline{r},\overline{x}'',\lambda}
		\phi)\Big)Z_{\overline{r},\overline{x}'',\lambda}
		\\
		\displaystyle \hspace{10.14mm}=
		\sum_{l=1}^{6}c_{l}\sum_{j=1}^{m}\Big[\Big(|x|^{-4}\ast |Z_{z_j,\lambda}|^{2}\Big)Z_{j,l}+2\Big(|x|^{-4}\ast (Z_{z_j,\lambda}
		Z_{j,l})\Big)Z_{z_j,\lambda}\Big]\hspace{4.14mm}\mbox{in}\hspace{1.14mm} \mathbb{R}^6,\\
		\displaystyle \phi\in H_{s}, \ \ \sum_{j=1}^{m}\int_{\mathbb{R}^{6}}\Big[\Big(|x|^{-4}\ast |Z_{z_j,\lambda}|^{2}\Big)Z_{j,l}\phi +2\Big(|x|^{-4}\ast |Z_{z_j,\lambda}Z_{j,l}|\Big)Z_{z_j,\lambda}\phi\Big] dx=0,~~l=1,2,\cdots,6,
	\end{array}
	\right.
\end{equation}
for certain constants $c_{l}$, has only trivial solution in $E$. This is true due to  Lemma \ref{C1}.
We can also conclude then
that for each $h$, problem \eqref{c1} admits a unique solution $\phi$ with
$$\|\phi\|_{\ast}\leq C \|h\|_{\ast\ast}. $$
In conclusion, we have the following Lemma:
\begin{lem}\label{C2}
There exists $m_0 > 0$ and a constant $C > 0$, independent of $m$, such that for all
$m \geq m_0$ and all $h\in L^{\infty}(\R^6)$, problem \eqref{c1} has a unique solution $\phi \equiv L_m(h)$. Besides,
\begin{equation}\label{c13}
	\|L_m(h)\|_{\ast}\leq C\|h\|_{\ast\ast},\quad |c_l|\leq \frac{C}{\lambda^{n_{l}}}\|h\|_{\ast\ast}.
\end{equation}
\end{lem}

Next we consider the following equation
\begin{equation}\label{c14}
\left\{\begin{array}{l}
	\displaystyle -\Delta (Z_{\overline{r},\overline{x}'',\lambda}+\phi)+ V(|x'|,x'')(Z_{\overline{r},\overline{x}'',\lambda}+\phi)
	-\Big(|x|^{-4}\ast |(Z_{\overline{r},\overline{x}'',\lambda}+\phi)|^{2}\Big)(Z_{\overline{r},\overline{x}'',\lambda}+\phi)\\
	\displaystyle \hspace{10.14mm}=\sum_{l=1}^{6}c_{l}\sum_{j=1}^{m}\Big[\Big(|x|^{-4}\ast |Z_{z_j,\lambda}|^{2}\Big)Z_{j,l}+2\Big(|x|^{-4}\ast (Z_{z_j,\lambda}
	Z_{j,l})\Big)Z_{z_j,\lambda}\Big]\hspace{4.14mm}\mbox{in}\hspace{1.14mm} \mathbb{R}^6,\\
	\displaystyle \phi\in H_{s}, \ \ \sum_{j=1}^{m}\int_{\mathbb{R}^{6}}\Big[\Big(|x|^{-4}\ast |Z_{z_j,\lambda}|^{2}\Big)Z_{j,l}\phi +2\Big(|x|^{-4}\ast |Z_{z_j,\lambda}Z_{j,l}|\Big)Z_{z_j,\lambda}\phi\Big] dx=0,~l=1,2,\dots,6.
\end{array}
\right.
\end{equation}

We can rewrite \eqref{c14} as
\begin{equation}\label{c16}
\aligned
-\Delta \phi+& V(|x'|,x'')\phi
-\Big(|x|^{-4}\ast |Z_{\overline{r},\overline{x}'',\lambda}|^{2}\Big)\phi-2\Big(|x|^{-4}\ast Z_{\overline{r},\overline{x}'',\lambda}\phi\Big)Z_{\overline{r},\overline{x}'',\lambda}\\
=&N(\phi)+l_{m}+\sum_{l=1}^{6}c_{l}\sum_{j=1}^{m}\Big[\Big(|x|^{-4}\ast |Z_{z_j,\lambda}|^{2}\Big)Z_{j,l}+2\Big(|x|^{-4}\ast (Z_{z_j,\lambda}
Z_{j,l})\Big)Z_{z_j,\lambda}\Big]\hspace{4.14mm}\mbox{in}\hspace{1.14mm} \mathbb{R}^6,
\endaligned
\end{equation}
where
$$\aligned
N(\phi)=&2\Big(|x|^{-4}\ast (Z_{\overline{r},\overline{x}'',\lambda}\phi)\Big)\phi+\Big(|x|^{-4}\ast |\phi|^{2}\Big)Z_{\overline{r},\overline{x}'',\lambda}+\Big(|x|^{-4}\ast |\phi|^{2}\Big)\phi
\endaligned$$
and
$$
l_{m}=\Big(|x|^{-4}\ast |Z_{\overline{r},\overline{x}'',\lambda}|^{2}\Big)Z_{\overline{r},\overline{x}'',\lambda}
-\sum_{j=1}^{m}\Big(|x|^{-4}\ast |Z_{z_j,\lambda}|^{2}\Big)Z_{z_j,\lambda}
-V(|x'|,x'')Z_{\overline{r},\overline{x}'',\lambda}
+Z_{\overline{r},\overline{x}'',\lambda}^{\ast}\Delta\xi+2\nabla\xi\nabla Z_{\overline{r},\overline{x}'',\lambda}^{\ast}.
$$

In order to use the contraction mapping theorem to prove that \eqref{c16} is uniquely solvable in
the set that $\|\phi\|_{\ast}$ is small, we need to estimate $N(\phi)$ and $l_m$.

\begin{lem}\label{C4}
There is a constant $C> 0$, such that
\begin{equation}\label{c17}
	\|N(\phi)\|_{\ast\ast}\leq C\|\phi\|_{\ast}^{2}.
\end{equation}
\end{lem}
\begin{proof}
Notice that
$$\aligned
\Big(|x|^{-4}\ast (Z_{\overline{r},\overline{x}'',\lambda}\phi)\Big)|\phi|
\leq &C \|Z_{\overline{r},\overline{x}'',\lambda}\|_{\ast}\|\phi\|_{\ast}^{2}
\lambda^{6}\bigg(|x|^{-4}\ast \Big(\sum_{j=1}^{m}
\frac{1}{(1+\lambda|x-z_{j}|)^{2+\tau}}\Big)^{2}\bigg)\sum_{j=1}^{m}
\frac{1}{(1+\lambda|x-z_{j}|)^{2+\tau}}\\
\leq &C \|\phi\|_{\ast}^{2}
\lambda^{4}\sum_{j=1}^{m}
\frac{1}{(1+\lambda|x-z_{j}|)^{2}}\sum_{j=1}^{m}
\frac{1}{(1+\lambda|x-z_{j}|)^{2+\tau}}\\
\leq &C \|\phi\|_{\ast}^{2}
\lambda^{4}\sum_{j=1}^{m}
\frac{1}{(1+\lambda|x-z_{j}|)^{4+\tau}},
\endaligned$$
where we applied the fact that, for any $1\leq j\leq m$, there is a constant $C>0$, such that
$$
|x|^{-4}\ast
\frac{\lambda^{2}}{(1+\lambda|x-z_{j}|)^{4+2\tau}}\leq C\frac{1}{(1+\lambda|x-z_{j}|)^{2}},
$$
whose proof is similar to Lemma \ref{B4}.
Similarly, we also have
$$
\Big(|x|^{-4}\ast |\phi|^{2}\Big)|Z_{\overline{r},\overline{x}'',\lambda}|
\leq C \|\phi\|_{\ast}^{2}
\lambda^{4}\sum_{j=1}^{m}
\frac{1}{(1+\lambda|x-z_{j}|)^{4+\tau}}
$$
and
$$
\Big(|x|^{-4}\ast |\phi|^{2}\Big)|\phi|
\leq C \|\phi\|_{\ast}^{3}
\lambda^{4}\sum_{j=1}^{m}
\frac{1}{(1+\lambda|x-z_{j}|)^{4+\tau}}.
$$
From the above, we have
$$
\|N(\phi)\|_{\ast\ast}\leq C\|\phi\|_{\ast}^{2}.
$$
\end{proof}

\begin{lem}\label{C5}
There is a constant $\varepsilon> 0$, such that
\begin{equation}\label{c18}
	\|l_{m}\|_{\ast\ast}\leq C(\frac{1}{\lambda})^{1+\varepsilon}.
\end{equation}
\end{lem}
\begin{proof}
Observe that
$$\aligned
l_{m}&=\sum_{i=1}^{m}\Big[\Big(|x|^{-4}\ast |Z_{z_i,\lambda}|^{2}\Big)\sum_{j\neq i}Z_{z_j,\lambda}\Big]+2\sum_{j=1}^{m}\sum_{i\neq j}\Big(|x|^{-4}\ast |Z_{z_j,\lambda}Z_{z_i,\lambda}|\Big)Z_{\overline{r},\overline{x}'',\lambda}
-V(|x'|,x'')Z_{\overline{r},\overline{x}'',\lambda}\\
&+(Z_{\overline{r},\overline{x}'',\lambda}^{\ast}\Delta\xi+2\nabla\xi\nabla Z_{\overline{r},\overline{x}'',\lambda}^{\ast})\\
:&=\Phi_{1}+\Phi_{2}+\Phi_{3}+\Phi_{4}.
\endaligned$$

First, we estimate the term  $\Phi_{1}$. Recall that $|x-z_{j}|\geq|x-z_{1}|, \ \forall x \in\Omega_{1}$, where
$$
\Omega_{j}=\left\{x=(x',x'')\in\mathbb{R}^{2}\times\mathbb{R}^{4}:\Big\langle\frac{x'}{|x'|},
\frac{z_{j}'}{|z_{j}'|}\Big\rangle\geq\cos\frac{\pi}{m}\right\}, \ \ j=1,\cdot\cdot\cdot,m.
$$
By Lemma \ref{P0}, we have
$$
|x|^{-4}\ast |Z_{z_i,\lambda}|^{2}\leq C|x|^{-4}\ast \Big|\frac{\lambda^{2}}{(1+\lambda^{2}|x-z_{i}|^{2})^{2}}\Big|^{2}\leq C \frac{\lambda^{2}}{(1+\lambda|x-z_{i}|)^{4}}.
$$
Taking $0 < \alpha \leq 4$, by Lemma \ref{B2}, we obtain that for any $x\in\Omega_{i}$, $i=1,\cdot\cdot\cdot,m$ and $j \neq i$
$$
\frac{1}{(1+\lambda|x-z_{i}|)^{4}}\frac{1}{(1+\lambda|x-z_{j}|)^{4}}\leq C
\frac{1}{(1+\lambda|x-z_{i}|)^{8-\alpha}}\frac{1}{|\lambda(z_{j}-z_{i})|^{\alpha}}.
$$
We can choose $\alpha> 3$ satisfying $4-\alpha\geq \tau$. Combining this with the fact
$$
\sum_{j\neq i}Z_{z_j,\lambda}\leq\sum_{j=1,\neq i}^{m}\frac{\lambda^{2}}{(1+\lambda|x-z_{j}|)^{4}},
$$
we have
$$
\Big(|x|^{-4}\ast |Z_{z_i,\lambda}|^{2}\Big)\sum_{j\neq i}Z_{z_j,\lambda}
\leq\frac{C\lambda^{4}}{(1+\lambda|x-z_{i}|)^{8-\alpha}}(\frac{m}{\lambda})^{\alpha}
\leq\frac{C\lambda^{4}}{(1+\lambda|x-z_{i}|)^{4+\tau}}(\frac{1}{\lambda})^{1+\varepsilon}.
$$
And so,
$$
|\Phi_{1}|
\leq C (\frac{1}{\lambda})^{1+\varepsilon}\lambda^{4}\sum_{i=1}^{m}\frac{1}{(1+\lambda|x-z_{i}|)^{4+\tau}}.
$$

To estimate $\Phi_{2}$,  by taking $0 < \alpha \leq 4$ and applying Lemma \ref{B2} again, we obtain that for any $x\in\Omega_{j}$ and $i\neq j$
$$
|Z_{z_j,\lambda}Z_{z_i,\lambda}|
\leq C
\frac{\lambda^{4}}{(1+\lambda|x-z_{j}|)^{8-\alpha}}\frac{1}{|\lambda(z_{j}-z_{i})|^{\alpha}}.
$$
We can choose $\alpha> 3$ satisfying $6-\alpha\geq \tau$. Then
$$\aligned
\frac{\lambda^{2}}{(1+\lambda|x-z_{j}|)^{4}}\sum_{i=1,\neq j}^{m}\frac{\lambda^{2}}{(1+\lambda|x-z_{i}|)^{4}}
&\leq\frac{C\lambda^{4}}{(1+\lambda|x-z_{j}|)^{8-\alpha}}(\frac{m}{\lambda})^{\alpha}.
\endaligned$$
By Lemma \ref{B3}, we have
$$\aligned
|\Phi_{2}|&\leq C\sum_{j=1}^{m}\Big(|x|^{-4}\ast \frac{\lambda^{4}}{(1+\lambda|x-z_{j}|)^{8-\alpha}}\Big)|Z_{\overline{r},\overline{x}'',\lambda}|\\
&\leq C(\frac{m}{\lambda})^{\alpha}\sum_{j=1}^{m}\frac{\lambda^{2}}{(1+\lambda|x-z_{j}|)^{6-\alpha}}
\sum_{i=1}^{m}\frac{\lambda^{2}}{(1+\lambda|x-z_{i}|)^{4}}\\
&\leq C (\frac{1}{\lambda})^{1+\varepsilon}\lambda^{4}\sum_{j=1}^{m}\frac{1}{(1+\lambda|x-z_{j}|)^{4+\tau}}.
\endaligned$$

For the remaining terms, it follows from Lemma 2.5 in \cite{PWY} that
$$
|\Phi_{3}|\leq C (\frac{1}{\lambda})^{1+\varepsilon}\lambda^{4}\sum_{j=1}^{m}\frac{1}{(1+\lambda|x-z_{j}|)^{4+\tau}},\ \
|\Phi_{4}|\leq C (\frac{1}{\lambda})^{1+\varepsilon}\lambda^{4}\sum_{j=1}^{m}\frac{1}{(1+\lambda|x-z_{j}|)^{4+\tau}}.
$$
So,
$$
\|l_{m}\|_{\ast\ast}\leq C(\frac{1}{\lambda})^{1+\varepsilon}.
$$
\end{proof}

We are ready to conclude the following estimates for the solution and the constants $c_l$.
\begin{lem}\label{C3}
	There is an integer $m_0 > 0$, such that for each $m \geq m_0$, $\lambda\in[L_0m^{2},L_1m^{2}]$,
	$\overline{r}\in[r_{0}-\theta,r_{0}+\theta]$, $\overline{x}''\in B_{\theta}(x_{0}'')$, where $\theta> 0$ is a fixed small constant, \eqref{c14} has a unique solution $\phi= \phi_{\overline{r},\overline{x}'',\lambda}\in H_{s}$,
	satisfying
	\begin{equation}\label{c15}
		\|\phi\|_{\ast}\leq C(\frac{1}{\lambda})^{1+\varepsilon}, \ \ |c_{l}|\leq C(\frac{1}{\lambda})^{1+n_{l}+\varepsilon},
	\end{equation}
	where $\varepsilon> 0$ is a small constant.
\end{lem}
\begin{proof}
First, we recall that $\lambda\in [L_0m^{2},L_1m^{2}]$. Set
$$\aligned
\mathcal{N}=\Big\{w:w\in C(\mathbb{R}^6)\cap H_{s},\|w\|_{\ast}\leq\frac{1}{\lambda},\ \
\sum_{j=1}^{m}\int_{\mathbb{R}^{6}}\Big[\Big(|x|^{-4}\ast |Z_{z_j,\lambda}|^{2}\Big)Z_{j,l}w +2\Big(|x|^{-4}\ast |Z_{z_j,\lambda}Z_{j,l}|\Big)Z_{z_j,\lambda}w\Big] dx=0.\Big\},
\endaligned$$
where $l=1,2,\cdots, 6$.
Then \eqref{c16} is equivalent to
\begin{equation}\label{c19}
\phi=\mathcal{A}(\phi)=:L_{m}(N(\phi))+L_{m}(l_{m}),
\end{equation}
where $L_m$ is defined in Lemma \ref{C2}. We will prove that $\mathcal{A}$ is a contraction map from $\mathcal{N}$ to $\mathcal{N}$.

First, we have
$$
\|\mathcal{A}\|_{\ast}\leq C(\|N(\phi)\|_{\ast\ast}+\|l_{m}\|_{\ast\ast})\leq C(\|\phi\|_{\ast}^{2}+(\frac{1}{\lambda})^{1+\varepsilon})\leq
\frac{1}{\lambda}.
$$
Hence, $\mathcal{A}$ maps $\mathcal{N}$ to $\mathcal{N}$.

Taking now $\phi_1$ and $\phi_2$ in $\mathcal{N}$, we see that
$$
\|\mathcal{A}(\phi_{1})-\mathcal{A}(\phi_{2})\|_{\ast}
=\|L_{m}(N(\phi_{1}))-L_{m}(N(\phi_{2}))\|_{\ast}\leq C\|N(\phi_{1})-N(\phi_{2})\|_{\ast\ast}.
$$
It is also easy to see that
$$\aligned
N(\phi_{1})-N(\phi_{2})
=&2\Big(|x|^{-4}\ast (Z_{\overline{r},\overline{x}'',\lambda}\phi_{1})\Big)(\phi_{1}-\phi_{2})+\Big(|x|^{-4}\ast (\phi_{1}+\phi_{2})(\phi_{1}-\phi_{2})\Big)Z_{\overline{r},\overline{x}'',\lambda}\\
&+\Big(|x|^{-4}\ast |\phi_{1}|^{2}\Big)(\phi_{1}-\phi_{2})+2\Big(|x|^{-4}\ast (Z_{\overline{r},\overline{x}'',\lambda}(\phi_{1}-\phi_{2}))\Big)\phi_{2}\\
&+\Big(|x|^{-4}\ast (\phi_{1}+\phi_{2})(\phi_{1}-\phi_{2})\Big)\phi_{2}.
\endaligned$$
Similar to the estimates in Lemma \ref{C4}, we have
$$\aligned
\|N(\phi_{1})-N(\phi_{2})\|_{\ast\ast}\leq &C\|\phi_{1}\|_{\ast} (\|\phi_{1}-\phi_{2}\|_{\ast})+C(\|\phi_{1}\|_{\ast}+\|\phi_{2}\|_{\ast}) (\|\phi_{1}-\phi_{2}\|_{\ast})+C\|\phi_{1}\|_{\ast}^{2} (\|\phi_{1}-\phi_{2}\|_{\ast})\\
&+C\|\phi_{2}\|_{\ast} (\|\phi_{1}-\phi_{2}\|_{\ast})+C(\|\phi_{1}\|_{\ast}+\|\phi_{2}\|_{\ast}) (\|\phi_{1}-\phi_{2}\|_{\ast})\|\phi_{2}\|_{\ast}.
\endaligned$$
Therefore
$$
\|\mathcal{A}(\phi_{1})-\mathcal{A}(\phi_{2})\|_{\ast}\leq C\|N(\phi_{1})-N(\phi_{2})\|_{\ast\ast}\leq \frac{1}{2}\|\phi_{1}-\phi_{2}\|_{\ast},
$$
which means that $\mathcal{A}$ is a contraction mapping from $\mathcal{N}$ into itself.
Thus we know that there exists a unique $\phi\in\mathcal{N}$ such
that \eqref{c19} holds. Moreover, by Lemmas \ref{C2}, \ref{C4} and \ref{C5}, we know
$$
\|\phi\|_{\ast}\leq C(\frac{1}{\lambda})^{1+\varepsilon}
$$
and the estimate of $c_l$ from \eqref{c13}.
\end{proof}

\section{Local Poho\v{z}aev identities methods}
In this section, we will look for suitable $(\overline{r},\overline{x}'',\lambda)$ such that the function $Z_{\overline{r},\overline{x}'',\lambda}+
\phi_{\overline{r},\overline{x}'',\lambda}$ obtained by reduction arguments is a solution of \eqref{CFL}. For this purpose, we need to establish new local Poho\v{z}aev identities for equation \eqref{CFL}.

\begin{lem}\label{D1}
Suppose that $(\overline{r},\overline{x}'',\lambda)$ satisfies
\begin{equation}\label{d1}
	\int_{D_{\rho}}\Big(-\Delta u_{m}+ V(|x'|,x'')u_{m}
	-\Big(|x|^{-4}\ast |u_{m}|^{2}\Big)u_{m}\Big)\langle x,\nabla u_{m}\rangle dx=0,
\end{equation}
\begin{equation}\label{d2}
	\int_{D_{\rho}}\Big(-\Delta u_{m}+ V(|x'|,x'')u_{m}
	-\Big(|x|^{-4}\ast |u_{m}|^{2}\Big)u_{m}\Big)\frac{\partial u_{m}}{\partial x_{i}} dx=0, i=3,\cdot\cdot\cdot,6,
\end{equation}
and
\begin{equation}\label{d3}
	\int_{\mathbb{R}^{6}}(-\Delta u_{m}+ V(|x'|,x'')u_{m}
	-\Big(|x|^{-4}\ast |u_{m}|^{2}\Big)u_{m})\frac{\partial Z_{\overline{r},\overline{x}'',\lambda}}{\partial \lambda} dx=0,
\end{equation}
where $u_{m}=Z_{\overline{r},\overline{x}'',\lambda}+
\phi_{\overline{r},\overline{x}'',\lambda}$ and $D_{\rho}=\{(r, x''):|(r, x'')-(r_0, x_0'' )|\leq\rho\}$ with $\rho\in(2\delta, 5\delta)$. Then $c_{i}=0, i=1,\cdot\cdot\cdot,6$.
\end{lem}
\begin{proof}
Since $Z_{\overline{r},\overline{x}'',\lambda}=0$ in $\mathbb{R}^{6}\backslash D_{\rho}$, we see that if \eqref{d1}-\eqref{d3} hold, then
\begin{equation}\label{d4}
	\sum_{l=1}^{6}c_{l}\sum_{j=1}^{m}\int_{\mathbb{R}^{6}}\Big[\Big(|x|^{-4}\ast |Z_{z_j,\lambda}|^{2}\Big)Z_{j,l}+2\Big(|x|^{-4}\ast (Z_{z_j,\lambda}
	Z_{j,l})\Big)Z_{z_j,\lambda}\Big]vdx=0,
\end{equation}
for $v=\langle x,\nabla u_{m}\rangle$, $\frac{\partial u_{m}}{\partial x_{i}}, i=3,\cdot\cdot\cdot,6$ and $\frac{\partial Z_{\overline{r},\overline{x}'',\lambda}}{\partial \lambda}$.

By direct calculations, we can prove
$$\aligned
&\left|\int_{\mathbb{R}^{6}}\Big(|x|^{-4}\ast |Z_{z_j,\lambda}|^{2}\Big)Z_{j,i}
		\frac{\partial Z_{z_l,\lambda}}{\partial x_{i}}dx\right|\\
&\hspace{4mm}\leq C\lambda^{2}\int_{\mathbb{R}^{6}}\int_{\mathbb{R}^{6}}\frac{1}{(1+|x-\lambda z_{j} |^{2})^{4}}\frac{1}{| x- y|^{4}}\frac{(y_{i}-\lambda\overline{x}_{i})^{2}}{(1+|y-\lambda z_{j} |^{2})^{3}}
\frac{1}{(1+|y-\lambda z_{l} |^{2})^{3}}dxdy,
\endaligned$$
where $ i=3,\cdot\cdot\cdot,6$ and $(\overline{x}_{3}, \overline{x}_{4}, \cdot\cdot\cdot,\overline{x}_{6})=\overline{x}$.
If $l=j$, we have
$$
\left|\int_{\mathbb{R}^{6}}\Big(|x|^{-4}\ast |Z_{z_j,\lambda}|^{2}\Big)Z_{j,i}
\frac{\partial Z_{z_l,\lambda}}{\partial x_{i}}dx\right|=O (\lambda^{2}).
$$
If $l\neq j$, similar to the arguments in the proof of Lemma \ref{C5}, we can prove the following result:
$$
\left|\int_{\mathbb{R}^{6}}\Big(|x|^{-4}\ast |Z_{z_j,\lambda}|^{2}\Big)Z_{j,i}
\left(\sum_{l=1,\neq j}^{m}\frac{\partial Z_{z_l,\lambda}}{\partial x_{i}}\right)dx\right|=O (\lambda^{2-\varepsilon}),
$$
for some $\varepsilon> 0$. So, we can get
\begin{equation}\label{d6}
	\sum_{j=1}^{m}\int_{\mathbb{R}^{6}}\Big[\Big(|x|^{-4}\ast |Z_{z_j,\lambda}|^{2}\Big)Z_{j,i}
	+2\Big(|x|^{-4}\ast (Z_{z_j,\lambda}
	Z_{j,i})\Big)Z_{z_j,\lambda}\Big]\frac{\partial Z_{\overline{r},\overline{x}'',\lambda}}{\partial x_{i}}dx=m(a_{1}+o(1))\lambda^{2}, \ i=3,\cdot\cdot\cdot, 6,
\end{equation}
for some constants $a_{1}\neq0$. Similarly, we have
\begin{equation}\label{d5}
	\sum_{j=1}^{m}\int_{\mathbb{R}^{6}}\Big[\Big(|x|^{-4}\ast |Z_{z_j,\lambda}|^{2}\Big)Z_{j,2}+2\Big(|x|^{-4}\ast (Z_{z_j,\lambda}
	Z_{j,2})\Big)Z_{z_j,\lambda} \Big]\langle x',\nabla_{x'} Z_{\overline{r},\overline{x}'',\lambda}\rangle dx=m(a_{2}+o(1))\lambda^{2},
\end{equation}
and
\begin{equation}\label{d7}
	\sum_{j=1}^{m}\int_{\mathbb{R}^{6}}\Big[\Big(|x|^{-4}\ast |Z_{z_j,\lambda}|^{2}\Big)Z_{j,1} +2\Big(|x|^{-4}\ast (Z_{z_j,\lambda}
	Z_{j,1})\Big)Z_{z_j,\lambda}\Big]\frac{\partial Z_{\overline{r},\overline{y}'',\lambda}}{\partial \lambda}dx=\frac{m}{\lambda^{2}}(a_{3}+o(1)),
\end{equation}
for some constants $a_{2}\neq0$ and $a_3>0$.

It is easy to check that
$$\aligned
&\int_{\mathbb{R}^{6}}\Big(|x|^{-4}\ast |Z_{z_j,\lambda}|^{2}\Big)Z_{j,1}\frac{\partial \phi_{\overline{r},\overline{x}'',\lambda}}{\partial x_{i}}dx\\
=&\int_{\mathbb{R}^6}\int_{\mathbb{R}^6}
\frac{|Z_{z_j,\lambda}(x)|^{2}
	\frac{\partial Z_{j,1}}{\partial y_{i}}(y)\phi_{\overline{r},\overline{x}'',\lambda}(y)}{|x-y|^{4}}dxdy
+\int_{\mathbb{R}^6}\int_{\mathbb{R}^6}
\frac{|Z_{z_j,\lambda}(x)|^{2}
	(x_{i}-y_{i})Z_{j,1}(y)\phi_{\overline{r},\overline{x}'',\lambda}(y)}{|x-y|^{6}}dxdy,
\endaligned$$
where $j=1,2,\cdot\cdot\cdot,m$ and $i=3,\cdot\cdot\cdot,6$. By direct calculations, by \eqref{c03} and \eqref{c15}, we can obtain
$$\aligned
&\left|\int_{\mathbb{R}^6}\int_{\mathbb{R}^6}
\frac{|Z_{z_j,\lambda}(x)|^{2}
	\frac{\partial Z_{j,1}}{\partial y_{i}}(y)\phi_{\overline{r},\overline{x}'',\lambda}(y)}{|x-y|^{4}}dxdy\right|\\
\leq &C\|\phi\|_{\ast}\int_{\mathbb{R}^6}\int_{\mathbb{R}^6}\frac{\lambda^{4}}{(1+\lambda^{2}|x-z_{j} |^{2})^{4}}\frac{1}{|x-y|^{4}}\sum_{k=1}^{m} \frac{\lambda^{2}}{(1+\lambda|y-z_{k}|)^{2+\tau}}\frac{\lambda^{3}|y-z_{j} |}{(1+\lambda^{2}|y-z_{j} |^{2})^{3}}dxdy\\
&+ C\|\phi\|_{\ast}\int_{\mathbb{R}^6}\int_{\mathbb{R}^6}\frac{\lambda^{4}}{(1+\lambda^{2}|x-z_{j} |^{2})^{4}}\frac{1}{|x-y|^{4}}\sum_{k=1}^{m} \frac{\lambda^{2}}{(1+\lambda|y-z_{k}|)^{2+\tau}}\frac{\lambda^{3}|y-z_{j} |}{(1+\lambda^{2}|y-z_{j} |^{2})^{4}}dxdy\\
\leq &C\lambda\|\phi\|_{\ast}\int_{\mathbb{R}^6}\int_{\mathbb{R}^6}\frac{1}{(1+|x-\lambda z_{j} |^{2})^{4}}\frac{1}{|x-y|^{4}}\frac{|y-\lambda z_{j} |}{(1+|y-\lambda z_{j} |^{2})^{5}}dxdy\\
&+ C\lambda\|\phi\|_{\ast}\int_{\mathbb{R}^6}\int_{\mathbb{R}^6}\frac{1}{(1+|x-\lambda z_{j} |^{2})^{4}}\frac{1}{|x-y|^{4}}\frac{|y-\lambda z_{j} |}{(1+|y-\lambda z_{j} |^{2})^{6}}dxdy\\
=&C\lambda\|\phi\|_{\ast}=O(\frac{1}{\lambda^{\varepsilon}}).
\endaligned$$
Similarly, we have
$$
\int_{\mathbb{R}^6}\int_{\mathbb{R}^6}
\frac{|Z_{z_j,\lambda}(x)|^{2}
	(x_{i}-y_{i})Z_{j,1}(y)\phi_{\overline{r},\overline{x}'',\lambda}(y)}{|x-y|^{6}}dxdy=O(\frac{1}{\lambda^{\varepsilon}}).
$$
And so,
$$
\int_{\mathbb{R}^{6}}\Big(|x|^{-4}\ast |Z_{z_j,\lambda}|^{2}\Big)Z_{j,1}\frac{\partial \phi_{\overline{r},\overline{x}'',\lambda}}{\partial x_{i}}dx=O(\frac{1}{\lambda^{\varepsilon}}).
$$
Similarly, we can also conclude that
$$
\int_{\mathbb{R}^{6}}\Big(|x|^{-4}\ast (Z_{z_j,\lambda}
Z_{j,1})\Big)Z_{z_j,\lambda}\frac{\partial \phi_{\overline{r},\overline{x}'',\lambda}}{\partial x_{i}}dx=O(\frac{1}{\lambda^{\varepsilon}}).
$$
Therefore
$$
c_{1}\sum_{j=1}^{m}\int_{\mathbb{R}^{6}}\Big[\Big(|x|^{-4}\ast |Z_{z_j,\lambda}|^{2}\Big)Z_{j,1}+2\Big(|x|^{-4}\ast (Z_{z_j,\lambda}
Z_{j,1})\Big)Z_{z_j,\lambda}\Big]\frac{\partial \phi_{\overline{r},\overline{x}'',\lambda}}{\partial x_{i}}dx=o(m|c_{1}|).
$$
Analogously, we can prove that
$$
\sum_{l=2}^{6}c_{l}\sum_{j=1}^{m}\int_{\mathbb{R}^{6}}\Big[\Big(|x|^{-4}\ast |Z_{z_j,\lambda}|^{2}\Big)Z_{j,l}+2\Big(|x|^{-4}\ast (Z_{z_j,\lambda}
Z_{j,l})\Big)Z_{z_j,\lambda}\Big]\frac{\partial \phi_{\overline{r},\overline{x}'',\lambda}}{\partial x_{i}}dx=o(m\lambda^{2})\sum_{l=2}^{6}|c_{l}|,
$$
and so
$$
\sum_{l=1}^{6}c_{l}\sum_{j=1}^{m}\int_{\mathbb{R}^{6}}\Big[\Big(|x|^{-4}\ast |Z_{z_j,\lambda}|^{2}\Big)Z_{j,l}+2\Big(|x|^{-4}\ast (Z_{z_j,\lambda}
Z_{j,l})\Big)Z_{z_j,\lambda}\Big]\frac{\partial \phi_{\overline{r},\overline{x}'',\lambda}}{\partial x_{i}}dx=o(m\lambda^{2})\sum_{l=2}^{6}|c_{l}|+o(m|c_{1}|),
$$
where $i=3,\cdot\cdot\cdot,6$. Repeating the same arguments, we can also obtain
$$
\sum_{l=1}^{6}c_{l}\sum_{j=1}^{m}\int_{\mathbb{R}^{6}}\Big[\Big(|x|^{-4}\ast |Z_{z_j,\lambda}|^{2}\Big)Z_{j,l}+2\Big(|x|^{-4}\ast (Z_{z_j,\lambda}
Z_{j,l})\Big)Z_{z_j,\lambda}\Big]\langle x,\nabla \phi_{\overline{r},\overline{x}'',\lambda}\rangle dx=o(m\lambda^{2})\sum_{l=2}^{6}|c_{l}|+o(m|c_{1}|).
$$
Therefore, from \eqref{d4}, we deduce that
\begin{equation}\label{d8}
	\sum_{l=1}^{6}c_{l}\sum_{j=1}^{m}\int_{\mathbb{R}^{6}}\Big[\Big(|x|^{-4}\ast |Z_{z_j,\lambda}|^{2}\Big)Z_{j,l}+2\Big(|x|^{-4}\ast (Z_{z_j,\lambda}
	Z_{j,l})\Big)Z_{z_j,\lambda}\Big]vdx=o(m\lambda^{2})\sum_{l=2}^{6}|c_{l}|+o(m|c_{1}|),
\end{equation}
holds for $v=\langle x,\nabla Z_{\overline{r},\overline{x}'',\lambda}\rangle$, $\frac{\partial Z_{\overline{r},\overline{x}'',\lambda}}{\partial x_{i}}$, $ i=3,\cdot\cdot\cdot,6$.

From
$$
\langle x,\nabla Z_{\overline{r},\overline{x}'',\lambda}\rangle=\langle x',\nabla_{x'} Z_{\overline{r},\overline{x}'',\lambda}\rangle+\langle x'',\nabla_{x''} Z_{\overline{r},\overline{x}'',\lambda}\rangle,
$$
we find
\begin{equation}\label{d9}
	\aligned
	&\sum_{l=1}^{6}c_{l}\sum_{j=1}^{m}\int_{\mathbb{R}^{6}}\Big[\Big(|x|^{-4}\ast |Z_{z_j,\lambda}|^{2}\Big)Z_{j,l} +2\Big(|x|^{-4}\ast (Z_{z_j,\lambda}
	Z_{j,l})\Big)Z_{z_j,\lambda}\Big]\langle x,\nabla Z_{\overline{r},\overline{x}'',\lambda}\rangle dx\\
	&\hspace{2mm}=c_{2}\sum_{j=1}^{m}\int_{\mathbb{R}^{6}}\Big[\Big(|x|^{-4}\ast |Z_{z_j,\lambda}|^{2}\Big)Z_{j,2}+2\Big(|x|^{-4}\ast (Z_{z_j,\lambda}
	Z_{j,2})\Big)Z_{z_j,\lambda}\Big]\langle x',\nabla_{x'} Z_{\overline{r},\overline{x}'',\lambda}\rangle dx\\
	&\hspace{4mm}+o(m\lambda^{2})\sum_{l=3}^{6}|c_{l}|+o(m|c_{1}|),
	\endaligned
\end{equation}
and
\begin{equation}\label{d10}
	\aligned
	&\sum_{l=1}^{6}c_{l}\sum_{j=1}^{m}\int_{\mathbb{R}^{6}}\Big[\Big(|x|^{-4}\ast |Z_{z_j,\lambda}|^{2}\Big)Z_{j,l} +2\Big(|x|^{-4}\ast (Z_{z_j,\lambda}
	Z_{j,l})\Big)Z_{z_j,\lambda}\Big]\frac{\partial Z_{\overline{r},\overline{x}'',\lambda}}{\partial x_{i}}dx\\
	&\hspace{2mm}=c_{i}\sum_{j=1}^{m}\int_{\mathbb{R}^{6}}\Big[\Big(|x|^{-4}\ast |Z_{z_j,\lambda}|^{2}\Big)Z_{j,i} +2\Big(|x|^{-4}\ast (Z_{z_j,\lambda}
	Z_{j,i})\Big)Z_{z_j,\lambda}\Big]\frac{\partial Z_{\overline{r},\overline{x}'',\lambda}}{\partial x_{i}}dx\\
	&\hspace{4mm}+o(m\lambda^{2})\sum_{l\neq1,i}^{6}|c_{l}|+o(m|c_{1}|), \ i=3,\cdot\cdot\cdot, 6.
	\endaligned
\end{equation}

Combining \eqref{d8}-\eqref{d10}, we are led to
$$
c_{2}\sum_{j=1}^{m}\int_{\mathbb{R}^{6}}\Big[\Big(|x|^{-4}\ast |Z_{z_j,\lambda}|^{2}\Big)Z_{j,l}+2\Big(|x|^{-4}\ast (Z_{z_j,\lambda}
Z_{j,l})\Big)Z_{z_j,\lambda}\Big]\langle x',\nabla_{x'} Z_{\overline{r},\overline{x}'',\lambda}\rangle dx=o(m\lambda^{2})\sum_{l=3}^{6}|c_{l}|+o(m|c_{1}|),
$$
and
$$
c_{i}\sum_{j=1}^{m}\int_{\mathbb{R}^{6}}\Big[\Big(|x|^{-4}\ast |Z_{z_j,\lambda}|^{2}\Big)Z_{j,i}+2\Big(|x|^{-4}\ast (Z_{z_j,\lambda}
Z_{j,i})\Big)Z_{z_j,\lambda}\Big]\frac{\partial Z_{\overline{r},\overline{x}'',\lambda}}{\partial x_{i}} dx=o(m\lambda^{2})\sum_{l\neq1,i}^{6}|c_{l}|+o(m|c_{1}|), \ i=3,\cdot\cdot\cdot, 6,
$$
which, together with \eqref{d5} and \eqref{d6}, imply
\begin{equation}\label{ci}
c_{i}=o(\frac{1}{\lambda^{2}})c_{1}, \ i=2,\cdot\cdot\cdot, 6.
\end{equation}

Now we have
$$\aligned
0=&\sum_{l=1}^{6}c_{l}\sum_{j=1}^{m}\int_{\mathbb{R}^{6}}\Big[\Big(|x|^{-4}\ast |Z_{z_j,\lambda}|^{2}\Big)Z_{j,l}+2\Big(|x|^{-4}\ast (Z_{z_j,\lambda}
Z_{j,l})\Big)Z_{z_j,\lambda}\Big]\frac{\partial Z_{\overline{r},\overline{x}'',\lambda}}{\partial \lambda} dx\\
=&c_{1}\sum_{j=1}^{m}\int_{\mathbb{R}^{6}}\Big[\Big(|x|^{-4}\ast |Z_{z_j,\lambda}|^{2}\Big)Z_{j,1}+2\Big(|x|^{-4}\ast (Z_{z_j,\lambda}
Z_{j,1})\Big)Z_{z_j,\lambda}\Big]\frac{\partial Z_{\overline{r},\overline{x}'',\lambda}}{\partial \lambda} dx+o(\frac{m}{\lambda^{2}})c_{1}\\
=&m(a_{3}+o(1))c_{1}+o(\frac{m}{\lambda^{2}})c_{1}.
\endaligned$$
So $c_1=0$.
\end{proof}

\begin{lem}\label{P2}
We have
$$
\int_{\mathbb{R}^6}\int_{\mathbb{R}^6}
\frac{|Z_{\overline{r},\overline{x}'',\lambda}^{\ast}(x)|^{2}
	Z_{\overline{r},\overline{x}'',\lambda}^{\ast}(y)\frac{\partial Z_{\overline{r},\overline{x}'',\lambda}^{\ast}}{\partial \overline{r}}(y)}{|x-y|^{4}}dxdy-\int_{\mathbb{R}^6}\int_{\mathbb{R}^6}
\frac{|Z_{\overline{r},\overline{x}'',\lambda}(x)|^{2}
	Z_{\overline{r},\overline{x}'',\lambda}(y)\frac{\partial Z_{\overline{r},\overline{x}'',\lambda}}{\partial \overline{r}}(y)}{|x-y|^{4}}dxdy=O(\frac{1}{\lambda^{2}}).
$$
\end{lem}
\begin{proof}
It is easy to check that
$$\aligned
&\int_{\mathbb{R}^6}\int_{\mathbb{R}^6}
\frac{|U_{z_j,\lambda}(x)|^{2}
	U_{z_j,\lambda}(y)\frac{\partial U_{z_j,\lambda}}{\partial \overline{r}}(y)}{|x-y|^{4}}dxdy-\int_{\mathbb{R}^6}\int_{\mathbb{R}^6}
\frac{|Z_{z_j,\lambda}(x)|^{2}
	Z_{z_j,\lambda}(y)\frac{\partial Z_{z_j,\lambda}}{\partial \overline{r}}(y)}{|x-y|^{4}}dxdy\\
=&\int_{\mathbb{R}^6}\int_{\mathbb{R}^6}
\frac{(1-\xi^{2}(x))|U_{z_j,\lambda}(x)|^{2}
	U_{z_j,\lambda}(y)\frac{\partial U_{z_j,\lambda}}{\partial \overline{r}}(y)}{|x-y|^{4}}dxdy+\int_{\mathbb{R}^6}\int_{\mathbb{R}^6}
\frac{|\xi U_{z_j,\lambda}(x)|^{2}
	(1-\xi^{2})U_{z_j,\lambda}(y)\frac{\partial U_{z_j,\lambda}}{\partial \overline{r}}(y)}{|x-y|^{4}}dxdy,
\endaligned$$
where $j=1,2, \cdot\cdot\cdot,m$. On the other hand, by the Hardy-Littlewood-Sobolev inequality and direct calculation, we have
$$\aligned
\int_{\mathbb{R}^6}\int_{\mathbb{R}^6}
&\frac{(1-\xi^{2}(x))|U_{z_j,\lambda}(x)|^{2}
	U_{z_j,\lambda}(y)\frac{\partial U_{z_j,\lambda}}{\partial \overline{r}}(y)}{|x-y|^{4}}dxdy\\
=&C\int_{\mathbb{R}^{6}}\int_{\mathbb{R}^{6}}\frac{(1-\xi^{2}(x+ z_{j}))\lambda^{4}}{(1+\lambda^{2}|x |^{2})^{4}}\frac{1}{|x-y|^{4}}\frac{\lambda^{6}(y_{1}\cos\frac{2(j-1)\pi}{m}
	+y_{2}\sin\frac{2(j-1)\pi}{m})}
{(1+\lambda^{2}|y|^{2})^{5}}dxdy\\
\leq &C \left(\int_{\mathbb{R}^{6}}\left[\frac{(1-\xi^{2}(x+z_{j}))\lambda^{4}}{(1+\lambda^{2}|x |^{2})^{4}}\right]^{\frac{3}{2}}dx\right)^{\frac{2}{3}}
\left(\int_{\mathbb{R}^{6}}\left[\frac{\lambda^{6}|y|^{2}}{(1+\lambda^{2}|y|^{2})^{5}}\right]^{\frac{3}{2}}dy\right)^{\frac{2}{3}}\\
=&O(\frac{1}{\lambda^{4}}),
\endaligned$$
where
$$
\int_{\mathbb{R}^{6}}\left[\frac{(1-\xi^{2}(x+z_{j}))\lambda^{4}}{(1+\lambda^{2}|x |^{2})^{4}}\right]^{\frac{3}{2}}dx
=O(\frac{1}{\lambda^{6}})
\ \mbox{and} \
\int_{\mathbb{R}^{6}}\left[\frac{\lambda^{6}|y|^{2}}{(1+\lambda^{2}|y|^{2})^{5}}\right]^{\frac{3}{2}}dy
=O(1).
$$
So,
$$
\int_{\mathbb{R}^6}\int_{\mathbb{R}^6}
\frac{(1-\xi^{2}(x))|U_{z_j,\lambda}(x)|^{2}
	U_{z_j,\lambda}(y)\frac{\partial U_{z_j,\lambda}}{\partial \overline{r}}(y)}{|x-y|^{4}}dxdy=O(\frac{1}{\lambda^{4}}).
$$
Analogously, it is easy to check that
$$
\int_{\mathbb{R}^6}\int_{\mathbb{R}^6}
\frac{|\xi U_{z_j,\lambda}(x)|^{2}
	(1-\xi^{2})U_{z_j,\lambda}(y)\frac{\partial U_{z_j,\lambda}}{\partial \overline{r}}(y)}{|x-y|^{4}}dxdy=O(\frac{1}{\lambda^{4}}).
$$
Thus, we have
$$
\int_{\mathbb{R}^6}\int_{\mathbb{R}^6}
\frac{|U_{z_j,\lambda}(x)|^{2}
	U_{z_j,\lambda}(y)\frac{\partial U_{z_j,\lambda}}{\partial \overline{r}}(y)}{|x-y|^{4}}dxdy-\int_{\mathbb{R}^6}\int_{\mathbb{R}^6}
\frac{|Z_{z_j,\lambda}(x)|^{2}
	Z_{z_j,\lambda}(y)\frac{\partial Z_{z_j,\lambda}}{\partial \overline{r}}(y)}{|x-y|^{4}}dxdy=O(\frac{1}{\lambda^{4}}),
$$
where $j=1,2, \cdot\cdot\cdot,m$. For other cases, we can obtain
$$
\int_{\mathbb{R}^6}\int_{\mathbb{R}^6}
\frac{U_{z_j,\lambda}(x)U_{z_i,\lambda}(x)
	U_{z_l,\lambda}(y)\frac{\partial U_{z_k,\lambda}}{\partial \overline{r}}(y)}{|x-y|^{4}}dxdy-\int_{\mathbb{R}^6}\int_{\mathbb{R}^6}
\frac{Z_{z_j,\lambda}(x)Z_{z_i,\lambda}(x)
	Z_{z_l,\lambda}(y)\frac{\partial Z_{z_k,\lambda}}{\partial \overline{r}}(y)}{|x-y|^{4}}dxdy=O(\frac{1}{\lambda^{4}}),
$$
where $j,i,l,k=1,2, \cdot\cdot\cdot,m$. Thus,
$$
\int_{\mathbb{R}^6}\int_{\mathbb{R}^6}
\frac{|Z_{\overline{r},\overline{x}'',\lambda}^{\ast}(x)|^{2}
	Z_{\overline{r},\overline{x}'',\lambda}^{\ast}(y)\frac{\partial Z_{\overline{r},\overline{x}'',\lambda}^{\ast}}{\partial \overline{r}}(y)}{|x-y|^{4}}dxdy-\int_{\mathbb{R}^6}\int_{\mathbb{R}^6}
\frac{|Z_{\overline{r},\overline{x}'',\lambda}(x)|^{2}
	Z_{\overline{r},\overline{x}'',\lambda}(y)\frac{\partial Z_{\overline{r},\overline{x}'',\lambda}}{\partial \overline{r}}(y)}{|x-y|^{4}}dxdy
=O(\frac{m^{4}}{\lambda^{4}})=O(\frac{1}{\lambda^{2}}).
$$
\end{proof}

\begin{lem}\label{P3}
We have
$$
\frac{\partial J(Z_{\overline{r},\overline{x}'',\lambda})}{\partial \overline{r}}=\frac{\partial J(Z_{\overline{r},\overline{x}'',\lambda}^{\ast})}{\partial \overline{r}}+O(\frac{1}{\lambda^{2}}),
$$
where
$$
J(u)=\frac{1}{2}\int_{\mathbb{R}^{6}}|\nabla u|^{2}dx+\frac{1}{2}\int_{\mathbb{R}^6} Vu^{2}dx
-\frac{1}{4}
\int_{\mathbb{R}^6}\int_{\mathbb{R}^6}
\frac{|u(x)|^{2}|u(y)|^{2}}{|x-y|^{4}}dxdy.
$$
\end{lem}
\begin{proof}
It is easy to check that
$$\aligned
\int_{\mathbb{R}^6}V
U_{z_j,\lambda}(y)\frac{\partial U_{z_j,\lambda}}{\partial \overline{r}}(y)dy-\int_{\mathbb{R}^6}
VZ_{z_j,\lambda}(y)\frac{\partial Z_{z_j,\lambda}}{\partial \overline{r}}(y)dy
=&\int_{\mathbb{R}^6}V(1-\xi^{2}(y))
U_{z_j,\lambda}(y)\frac{\partial U_{z_j,\lambda}}{\partial \overline{r}}(y)dy\\
\leq &C\int_{\mathbb{R}^6}(1-\xi^{2}(y))
U_{z_j,\lambda}(y)\frac{\partial U_{z_j,\lambda}}{\partial \overline{r}}(y)dy,
\endaligned$$
where $j=1,2, \cdot\cdot\cdot,m$. On the other hand, direct calculation gives
\begin{equation}\label{z1}
	\aligned
	\left|\int_{\mathbb{R}^6}(1-\xi^{2}(y))
	U_{z_j,\lambda}(y)\frac{\partial U_{z_j,\lambda}}{\partial \overline{r}}(y)dy\right|
&\leq C\int_{\mathbb{R}^{6}}\frac{(1-\xi^{2}(y))\lambda^{2}}{(1+\lambda^{2}|y-z_{j} |^{2})^{2}}\frac{\lambda^{4}|y-z_{j} |}
	{(1+\lambda^{2}|y-z_{j} |^{2})^{3}}dy\\
	&\leq C\int_{\mathbb{R}^{6}}\frac{(1-\xi^{2}(y+ z_{j}))\lambda^{2}}{(1+\lambda^{2}|y |^{2})^{2}}\frac{\lambda^{4}|y|}{(1+\lambda^{2}|y|^{2})^{3}}dy\\
	&\leq C\int_{\mathbb{R}^{6}\setminus B_{\delta-\vartheta} (0)}\frac{\lambda^{2}}{(1+\lambda^{2}|y |^{2})^{2}}\frac{\lambda^{4}|y|}{(1+\lambda^{2}|y|^{2})^{3}}dy\\
	&
	=O(\frac{1}{\lambda^{4}}).
	\endaligned
\end{equation}
Analogously, for other cases, we can obtain
$$
\int_{\mathbb{R}^6}V
U_{z_i,\lambda}(y)\frac{\partial U_{z_j,\lambda}}{\partial \overline{r}}(y)dy-\int_{\mathbb{R}^6}
VZ_{z_i,\lambda}(y)\frac{\partial Z_{z_j,\lambda}}{\partial \overline{r}}(y)dy=O(\frac{1}{\lambda^{4}}),
$$
where $j,i=1,2, \cdot\cdot\cdot,m$. Thus, we have
\begin{equation}\label{z2}
	\int_{\mathbb{R}^{6}}V\Big(Z_{\overline{r},\overline{x}'',\lambda}^{\ast}\frac{\partial Z_{\overline{r},\overline{x}'',\lambda}^{\ast}}{\partial \overline{r}}-Z_{\overline{r},\overline{x}'',\lambda}^{\ast}\frac{\partial Z_{\overline{r},\overline{x}'',\lambda}^{\ast}}{\partial \overline{r}}\Big)dx=O(\frac{m^{2}}{\lambda^{4}}).
\end{equation}

We observe that
$$\aligned
\int_{\mathbb{R}^{6}}&\Delta Z_{\overline{r},\overline{x}'',\lambda}^{\ast}\frac{\partial Z_{\overline{r},\overline{x}'',\lambda}^{\ast}}{\partial \overline{r}}dx-\int_{\mathbb{R}^{6}}\Delta Z_{\overline{r},\overline{x}'',\lambda}\frac{\partial Z_{\overline{r},\overline{x}'',\lambda}}{\partial \overline{r}}dx\\
&=\int_{\mathbb{R}^{6}}\Delta Z_{\overline{r},\overline{x}'',\lambda}^{\ast}\frac{\partial Z_{\overline{r},\overline{x}'',\lambda}^{\ast}}{\partial \overline{r}}dx-\int_{\mathbb{R}^{6}}\xi(\xi\Delta Z_{\overline{r},\overline{x}'',\lambda}^{\ast}+Z_{\overline{r},\overline{x}'',\lambda}^{\ast}\Delta\xi +2\nabla\xi\nabla Z_{\overline{r},\overline{x}'',\lambda}^{\ast})\frac{\partial Z_{\overline{r},\overline{x}'',\lambda}^{\ast}}{\partial \overline{r}}dx.
\endaligned$$
By Lemma \ref{P2}, we know
$$
\int_{\mathbb{R}^{6}}(1-\xi^{2})\Delta Z_{\overline{r},\overline{x}'',\lambda}^{\ast}\frac{\partial Z_{\overline{r},\overline{x}'',\lambda}^{\ast}}{\partial \overline{r}}dx
=\sum_{j=1}^{m}\int_{\mathbb{R}^6}\int_{\mathbb{R}^6}
\frac{(1-\xi^{2}(x))|U_{z_j,\lambda}(x)|^{2}
	U_{z_j,\lambda}(y)\frac{\partial Z_{\overline{r},\overline{x}'',\lambda}^{\ast}}{\partial \overline{r}}(y)}{|x-y|^{4}}dxdy=O(\frac{m^{2}}{\lambda^{4}}).
$$
Similar to the estimates of \eqref{z1}, we have
$$
\int_{\mathbb{R}^{6}}\xi Z_{\overline{r},\overline{x}'',\lambda}^{\ast}\Delta\xi \frac{\partial Z_{\overline{r},\overline{x}'',\lambda}^{\ast}}{\partial \overline{r}}dx=O(\frac{m^{2}}{\lambda^{4}}),
$$
and
$$
\int_{\mathbb{R}^{6}}\xi\nabla\xi\nabla Z_{\overline{r},\overline{x}'',\lambda}^{\ast}\frac{\partial Z_{\overline{r},\overline{x}'',\lambda}^{\ast}}{\partial \overline{r}}dx\leq C\sum_{j=1}^{m}\int_{\mathbb{R}^{6}}\frac{\xi|\nabla\xi|\lambda^{4}|y-z_{j} |}{(1+\lambda^{2}|y-z_{j} |^{2})^{3}}\frac{\partial Z_{\overline{r},\overline{x}'',\lambda}^{\ast}}{\partial \overline{r}}dy=O(\frac{m^{2}}{\lambda^{4}}).
$$
So, we have proved
$$
\int_{\mathbb{R}^{6}}\Delta Z_{\overline{r},\overline{x}'',\lambda}^{\ast}\frac{\partial Z_{\overline{r},\overline{x}'',\lambda}^{\ast}}{\partial \overline{r}}dx-\int_{\mathbb{R}^{6}}\Delta Z_{\overline{r},\overline{x}'',\lambda}\frac{\partial Z_{\overline{r},\overline{x}'',\lambda}}{\partial \overline{r}}dx=O(\frac{m^{2}}{\lambda^{4}})=O(\frac{1}{\lambda^{3}}).
$$

Combining this with \eqref{z1}, \eqref{z2} and Lemma \ref{P2}, we can obtain the result.
\end{proof}

\begin{lem}\label{D2}
We have
$$
\frac{\partial J(Z_{\overline{r},\overline{x}'',\lambda})}{\partial \overline{r}}=m\Big(\frac{B_{1}}{\lambda^{2}}\frac{\partial V(\overline{r},\overline{x}'')}{\partial \overline{r}}+\sum_{j=2}^{m}
\frac{B_{2}}{\overline{r}\lambda^{4}|z_{1}-z_{j}|^{4}}+O(\frac{1}{\lambda^{1+\varepsilon}})\Big),
$$
where $B_j$, $j=1, 2$ are some positive constants.
\end{lem}
\begin{proof}
We know
\begin{equation}\label{ExpenL4.4}
\aligned
\frac{\partial J(Z_{\overline{r},\overline{x}'',\lambda}^{\ast})}{\partial \overline{r}}=&\int_{\mathbb{R}^{6}}V(x)Z_{\overline{r},\overline{x}'',\lambda}^{\ast}\frac{\partial Z_{\overline{r},\overline{x}'',\lambda}^{\ast}}{\partial \overline{r}}dx-\int_{\mathbb{R}^{6}}\Big(|x|^{-4}\ast |Z_{\overline{r},\overline{x}'',\lambda}^{\ast}|^{2}\Big)Z_{\overline{r},\overline{x}'',\lambda}^{\ast}
\frac{\partial Z_{\overline{r},\overline{x}'',\lambda}^{\ast}}{\partial \overline{r}}dx\\
&+\sum_{j=1}^{m}\int_{\mathbb{R}^{6}}\Big(|x|^{-4}\ast |U_{z_j,\lambda}|^{2}\Big)U_{z_j,\lambda}\frac{\partial Z_{\overline{r},\overline{x}'',\lambda}^{\ast}}{\partial \overline{r}}dx\\
:=&\cp_1+\cp2.
\endaligned
\end{equation}

To estimate the first part $\cp_1$ in \eqref{ExpenL4.4}, integrating by parts, we notice that it is easy to show
$$\aligned
&\int_{\mathbb{R}^6}V(x)
U_{z_j,\lambda}(x)\frac{\partial U_{z_j,\lambda}}{\partial \overline{r}}(x)dx
= -\int_{\mathbb{R}^{6}}|U_{0,\lambda}(x)|^{2}\frac{\partial V(x+z_{j})}{\partial \overline{r}} dx
-\int_{\mathbb{R}^6}V(x)
U_{z_j,\lambda}(x)\frac{\partial U_{z_j,\lambda}}{\partial \overline{r}}(x)dx,
\endaligned$$
and so,
$$\aligned
\int_{\mathbb{R}^6}V(x)
U_{z_j,\lambda}(x)\frac{\partial U_{z_j,\lambda}}{\partial \overline{r}}(x)dx
= &C\int_{\mathbb{R}^{6}}|U_{0,\lambda}(x)|^{2}\frac{\partial V(z_{j})}{\partial \overline{r}} dx+C\int_{\mathbb{R}^{6}}|U_{0,\lambda}(x)|^{2}\frac{\partial (V(x+z_{j})-V(z_{j}))}{\partial \overline{r}} dx\\
= &\frac{C}{\lambda^{2}}\frac{\partial V(\overline{r},\overline{x}'')}{\partial \overline{r}}\int_{\mathbb{R}^{6}}|U_{0,1}(x)|^{2}dx+O(\frac{1}{\lambda^{2}}).\\
\endaligned$$
Combining this and the fact
$$\aligned
\int_{\mathbb{R}^6}V(x)
U_{z_1,\lambda}(x)\sum_{j=2}^{m}\frac{\partial U_{z_j,\lambda}}{\partial \overline{r}}(x)dx
\leq &\sum_{j=2}^{m}\frac{C}{\lambda}\int_{\mathbb{R}^{6}}\frac{1}{(1+|x-\lambda z_{1} |)^{4}}\frac{1}
{(1+|x-\lambda z_{j} |)^{5}}dx\\
\leq&\frac{C}{\lambda}\sum_{j=2}^{m}
	\frac{1}{\lambda^{2}|z_{1}-z_{j}|^{2}}=O(\frac{m^{2}}{\lambda^{3}})=O(\frac{1}{\lambda^{2}}),
\endaligned$$
we have
\begin{equation}\label{e00}\aligned
	\cp_1=\int_{\mathbb{R}^{6}}V(x)Z_{\overline{r},\overline{x}'',\lambda}^{\ast}\frac{\partial Z_{\overline{r},\overline{x}'',\lambda}^{\ast}}{\partial \overline{r}}dx=&\frac{Cm}{\lambda^{2}}\frac{\partial V(\overline{r},\overline{x}'')}{\partial \overline{r}}\int_{\mathbb{R}^{6}}|U_{0,1}(x)|^{2}dx+O(\frac{m}{\lambda^{2}})\\
=&\frac{Cm}{\lambda^{2}}\frac{\partial V(\overline{r},\overline{x}'')}{\partial \overline{r}}\int_{\mathbb{R}^{6}}|U_{0,1}(x)|^{2}dx+O(\frac{1}{\lambda^{1+\varepsilon}}).
\endaligned\end{equation}

To estimate the second part $\cp_2$ in \eqref{ExpenL4.4}, we observe that
\begin{equation}\label{5.1}
\aligned
	&\int_{\mathbb{R}^{6}}\Bigg[\Big(|x|^{-4}\ast |Z_{\overline{r},\overline{x}'',\lambda}^{\ast}|^{2}\Big)Z_{\overline{r},\overline{x}'',\lambda}^{\ast}-\sum_{j=1}^{m}\Big(|x|^{-4}\ast |U_{z_j,\lambda}|^{2}\Big)U_{z_j,\lambda}\Bigg] \frac{\partial Z_{\overline{r},\overline{x}'',\lambda}^{\ast}}{\partial \overline{r}}dx\\
&=2m\int_{\mathbb{R}^{6}}\Big(|x|^{-4}\ast |U_{z_1,\lambda}\sum_{i=2}^{m}U_{z_i,\lambda}|\Big)Z_{\overline{r},\overline{x}'',\lambda}^{\ast}\frac{\partial Z_{\overline{r},\overline{x}'',\lambda}^{\ast}}{\partial \overline{r}}dx
	+m\int_{\mathbb{R}^{6}}\Big[\Big(|x|^{-4}\ast |U_{z_1,\lambda}|^{2}\Big)\sum_{i=2}^{m}U_{z_i,\lambda}\Big]\frac{\partial Z_{\overline{r},\overline{x}'',\lambda}^{\ast}}{\partial \overline{r}}dx\\
:&=2m \cq_1+m \cq_2.
\endaligned\end{equation}

First, for the integral part $\cq2$, if $i=2,\cdot\cdot\cdot, m$, applying Lemmas \ref{P0} and \ref{P1}, we have
$$\aligned
\left|\int_{\mathbb{R}^{6}}\Big[\Big(|x|^{-4}\ast |U_{z_1,\lambda}|^{2}\Big)U_{z_i,\lambda}\Big]\frac{\partial U_{z_1,\lambda}}{\partial \overline{r}}dx\right|
\leq&\left|C\int_{\mathbb{R}^{6}}\frac{\lambda^{2}}{(1+\lambda^{2}|x-z_{1} |^{2})^{2}}\frac{\lambda^{2}}{(1+\lambda^{2}|x-z_{i} |^{2})^{2}}\frac{\lambda^{4}(x_{1}-\overline{r})}{(1+\lambda^{2}|x-z_{1} |^{2})^{3}}dx\right|\\
\leq&C\int_{\mathbb{R}^{6}}\frac{\lambda}{(1+|x-\lambda z_{1} |)^{8}}\frac{1}{(1+|x-\lambda z_{i} |)^{5}}dx\\
\leq&
\frac{C}{\lambda^{4}|z_{1}-z_{i}|^{5}}.
\endaligned$$
Consequently,
\begin{equation}\label{5.2}
\aligned
\left|\sum_{i=2}^{m}\int_{\mathbb{R}^{6}}\Big[\Big(|x|^{-4}\ast |U_{z_1,\lambda}|^{2}\Big)U_{z_i,\lambda}\Big]\frac{\partial U_{z_1,\lambda}}{\partial \overline{r}}dx\right|
\leq\sum_{i=2}^{m}\frac{C}{\lambda^{4}|z_{1}-z_{i}|^{5}}=\frac{Cm^{5}}{\lambda^{4}}=O(\frac{1}{\lambda^{1+\varepsilon}}).
\endaligned\end{equation}

Second, by Lemma \ref{P0}, we have
$$\aligned
&\int_{\mathbb{R}^{6}}\Big(|x|^{-4}\ast |U_{z_1,\lambda}|^{2}\Big)U_{z_i,\lambda}\frac{\partial U_{z_i,\lambda}}{\partial \overline{r}}dx\\
=&\frac{C}{\overline{r}}\int_{\mathbb{R}^{6}}\frac{\lambda^{2}}{(1+\lambda^{2}|x-z_{1} |^{2})^{2}}\frac{\lambda^{2}}{(1+\lambda^{2}|x-z_{i} |^{2})^{2}}
\frac{\lambda^{4}[(x_{1}-\overline{r}\cos\frac{2(i-1)\pi}{m})x_{1}
+(x_{2}-\overline{r}\sin\frac{2(i-1)\pi}{m})x_{2}]}
{(1+\lambda^{2}|x-z_{i} |^{2})^{3}}dx\\
&-\frac{C}{\overline{r}}\int_{\mathbb{R}^{6}}\frac{\lambda^{2}}{(1+\lambda^{2}|x-z_{1} |^{2})^{2}}\frac{\lambda^{2}}{(1+\lambda^{2}|x-z_{i} |^{2})^{2}}
\frac{\lambda^{4}|x-z_{i} |^{2}}
{(1+\lambda^{2}|x-z_{i} |^{2})^{3}}dx\\
&+\frac{C}{\overline{r}}\int_{\mathbb{R}^{6}}\frac{\lambda^{2}}{(1+\lambda^{2}|x-z_{1} |^{2})^{2}}\frac{\lambda^{2}}{(1+\lambda^{2}|x-z_{i} |^{2})^{2}}
\frac{\lambda^{4}(x''-\overline{x}'')^{2}}
{(1+\lambda^{2}|x-z_{i} |^{2})^{3}}dx\\
:=&\frac{C}{\overline{r}}\Psi_{1}-\frac{C}{\overline{r}}\Psi_{2}+\frac{C}{\overline{r}}\Psi_{3},
\endaligned$$
where $i=2,\cdot\cdot\cdot, m$. In fact,
$$\aligned
\Psi_{1}&\leq\left|\int_{\mathbb{R}^{6}}\frac{\lambda^{4}}{(1+\lambda^{2}|x-z_{i} |^{2})^{4}}
\frac{\lambda^{4}(x_{1}-\overline{r})x_{1}}
{(1+\lambda^{2}|x-z_{1} |^{2})^{3}}dx\right|+\left|\int_{\mathbb{R}^{6}}\frac{\lambda^{2}}{(1+\lambda^{2}|x-z_{1} |^{2})^{2}}\frac{\lambda^{4}}{(1+\lambda^{2}|x-z_{i} |^{2})^{4}}dx\right|\\
&\leq \frac{C}{\lambda^{4}|z_{1}-z_{i}|^{5}}+ \frac{C}{\lambda^{4}|z_{1}-z_{i}|^{4}}.
\endaligned$$
By \eqref{p1}, we  also have
$$
\Psi_{2}
=\int_{\mathbb{R}^{6}}\frac{\lambda^{4}}{(1+\lambda^{2}|x-z_{1} |^{2})^{2}}\frac{\lambda^{2}}{(1+\lambda^{2}|x-z_{i} |^{2})^{4}}dx
-\int_{\mathbb{R}^{6}}\frac{\lambda^{4}}{(1+\lambda^{2}|x-z_{1} |^{2})^{2}}\frac{\lambda^{2}}{(1+\lambda^{2}|x-z_{i} |^{2})^{5}}dx\\
=\frac{C}{\lambda^{4}|z_{1}-z_{i}|^{4}},
$$
and
$$
\Psi_{3}\leq\frac{C}{\lambda^{4}|z_{1}-z_{i}|^{4}}.
$$
Therefore, we obtain
\begin{equation}\label{5.3}
\aligned
\sum_{i=2}^{m}\int_{\mathbb{R}^{6}}\Big[\Big(|x|^{-4}\ast |U_{z_1,\lambda}|^{2}\Big)U_{z_i,\lambda}\Big]\frac{\partial U_{z_i,\lambda}}{\partial \overline{r}}dx&\leq\sum_{i=2}^{m}\frac{C}{\lambda^{4}|z_{1}-z_{i}|^{4}}+\sum_{i=2}^{m}\frac{C}{\lambda^{4}|z_{1}-z_{i}|^{5}}-\sum_{i=2}^{m}\frac{C}{\overline{r}}\frac{1}{\lambda^{4}|z_{1}-z_{i}|^{4}}\\
&=O(\frac{1}{\lambda^{1+\varepsilon}})-\sum_{i=2}^{m}\frac{C}{\overline{r}}\frac{1}{\lambda^{4}|z_{1}-z_{i}|^{4}}.
\endaligned\end{equation}
Third, for $j\neq1, i$, by Lemma \ref{B2}, we have
$$\aligned
\int_{\mathbb{R}^{6}}&\Big[\Big(|x|^{-4}\ast |U_{z_1,\lambda}|^{2}\Big)U_{z_i,\lambda}\Big]\frac{\partial U_{z_j,\lambda}}{\partial \overline{r}}dx\\
&\leq C\int_{\mathbb{R}^{6}}\frac{\lambda^{2}}{(1+\lambda^{2}|x-z_{1}|^{2})^{2}}
\frac{\lambda^{2}}{(1+\lambda^{2}|x-z_{i} |^{2})^{2}}\frac{\lambda^{4}|x-z_{j}|}
{(1+\lambda^{2}|x-z_{j} |^{2})^{3}}dx\\
&\leq \frac{C\lambda}{\lambda^{4}|z_{1}-z_{i}|^{4}}\int_{\mathbb{R}^{6}}
\Big(\frac{1}{(1+|x-\lambda z_{1}|)^{4}}+
\frac{1}{(1+|x-\lambda z_{i} |)^{4}}\Big)\frac{1}
{(1+|x-\lambda z_{j} |)^{5}}dx\\
&\leq\frac{C\lambda}{\lambda^{4}|z_{1}-z_{i}|^{4}}\frac{C}{\lambda^{2}|z_{1}-z_{j}|^{2}}
\int_{\mathbb{R}^{6}}
\Big(\frac{1}{(1+|x-\lambda z_{1}|)^{7}}+ \frac{1}
{(1+|x-\lambda z_{j} |)^{7}}\Big)dx\\
&\hspace{4mm}+\frac{C\lambda}{\lambda^{4}|z_{1}-z_{i}|^{4}}\frac{C}{\lambda^{2}|z_{i}-z_{j}|^{2}}
\int_{\mathbb{R}^{6}}
\Big(\frac{1}{(1+|x-\lambda z_{i}|)^{7}}+ \frac{1}
{(1+|x-\lambda z_{j} |)^{7}}\Big)dx\\
&\leq\frac{ C}{\lambda^{4}|z_{1}-z_{i}|^{4}}\frac{1}{\lambda|z_{1}-z_{j}|^{2}}
+\frac{ C}{\lambda^{4}|z_{1}-z_{i}|^{4}}\frac{1}{\lambda|z_{i}-z_{j}|^{2}}.
\endaligned$$
Thus, we can get
$$\aligned
\sum_{i=2}^{m}\sum_{j=2,\neq i}^{m}\int_{\mathbb{R}^{6}}&\Big[\Big(|x|^{-4}\ast |U_{z_1,\lambda}|^{2}\Big)U_{z_i,\lambda}\Big]\frac{\partial U_{z_j,\lambda}}{\partial \overline{r}}dx\\
\leq&\sum_{i=2}^{m}\sum_{j=2,\neq i}^{m}\frac{C}{\lambda^{4}|z_{1}-z_{i}|^{4}}\frac{1}{\lambda|z_{1}-z_{j}|^{2}}
+\sum_{i=2}^{m}\sum_{j=2,\neq i}^{m}\frac{C}{\lambda^{4}|z_{1}-z_{i}|^{4}}\frac{1}{\lambda|z_{i}-z_{j}|^{2}}\\
=&\frac{Cm^{6}}{\lambda^{5}}=O(\frac{1}{\lambda^{1+\varepsilon}}).
\endaligned$$
Combining this and \eqref{5.2}, \eqref{5.3}, we have
\begin{equation}\label{5.4}
\int_{\mathbb{R}^{6}}\Big[\Big(|x|^{-4}\ast |U_{z_1,\lambda}|^{2}\Big)\sum_{i=2}^{m}U_{z_i,\lambda}\Big]\frac{\partial Z_{\overline{r},\overline{x}'',\lambda}^{\ast}}{\partial \overline{r}}dx=-\sum_{j=2}^{m}
	\frac{B_{2}}{\overline{r}\lambda^{4}|z_{1}-z_{j}|^{4}}+O(\frac{1}{\lambda^{1+\varepsilon}}),
\end{equation}
for some constant $B_2>0$.

Next, we are going to estimate $\cq_1$. Observe that
\begin{equation}\label{5.5}
\aligned
\int_{\mathbb{R}^{6}}&\Big(|x|^{-4}\ast |U_{z_1,\lambda}\sum_{i=2}^{m}U_{z_i,\lambda}|\Big)Z_{\overline{r},\overline{x}'',\lambda}^{\ast}\frac{\partial Z_{\overline{r},\overline{x}'',\lambda}^{\ast}}{\partial \overline{r}}dx\\
&=\sum_{j=1}^{m}\int_{\mathbb{R}^{6}}\Big(|x|^{-4}\ast |U_{z_1,\lambda}\sum_{i=2}^{m}U_{z_i,\lambda}|\Big)U_{z_j,\lambda}\frac{\partial U_{z_j,\lambda}}{\partial \overline{r}}dx
+\int_{\mathbb{R}^{6}}\Big(|x|^{-4}\ast |U_{z_1,\lambda}\sum_{i=2}^{m}U_{z_i,\lambda}|\Big)U_{z_1,\lambda}\sum_{j=2}^{m}\frac{\partial U_{z_j,\lambda}}{\partial \overline{r}}dx\\
&\hspace{4mm}+(m-1)\sum_{j=1,\neq2}^{m}\int_{\mathbb{R}^{6}}\Big(|x|^{-4}\ast |U_{z_1,\lambda}\sum_{i=2}^{m}U_{z_i,\lambda}|\Big)U_{z_2,\lambda}\frac{\partial U_{z_j,\lambda}}{\partial \overline{r}}dx.
\endaligned\end{equation}

Integrate by parts, we have
$$\aligned
&\left|\int_{\mathbb{R}^{6}}\Big(|x|^{-4}\ast |U_{z_1,\lambda}U_{z_i,\lambda}|\Big)U_{z_j,\lambda}\frac{\partial U_{z_j,\lambda}}{\partial \overline{r}}dx\right|\\
\leq &C
\left|\int_{\mathbb{R}^{6}}\int_{\mathbb{R}^{6}}U_{z_1,\lambda}(x)U_{z_i,\lambda}(x)
\frac{[(x_{1}-y_{1})\cos\frac{2(j-1)\pi}{m}
+(x_{2}-y_{2})\sin\frac{2(j-1)\pi}{m}]}{|x-y|^{6}}|U_{z_j,\lambda}(y)|^{2}dxdy\right|\\
\leq&\left|\int_{\mathbb{R}^{6}}\Big(|y|^{-4}\ast |U_{z_j,\lambda}|^{2}\Big)U_{z_1,\lambda}\frac{\lambda^{4}[(y_{1}-\overline{r}\cos\frac{2(i-1)\pi}{m})\cos\frac{2(j-1)\pi}{m}
+(y_{2}-\overline{r}\sin\frac{2(i-1)\pi}{m})\sin\frac{2(j-1)\pi}{m}]}
{(1+\lambda^{2}|y-z_{i} |^{2})^{3}}dy\right|\\
&+\left|\int_{\mathbb{R}^{6}}\Big(|y|^{-4}\ast |U_{z_j,\lambda}|^{2}\Big)U_{z_i,\lambda}\frac{\lambda^{4}[(y_{1}-\overline{r})\cos\frac{2(j-1)\pi}{m}
+y_{2}\sin\frac{2(j-1)\pi}{m}]}
{(1+\lambda^{2}|y-z_{1} |^{2})^{3}}dy\right|,
\endaligned$$
where $i=2,\cdot\cdot\cdot, m$ and $j=1,\cdot\cdot\cdot, m$.
If $j=1$,  we have
$$\aligned
&\left|\int_{\mathbb{R}^{6}}\Big(|x|^{-4}\ast |U_{z_1,\lambda}U_{z_i,\lambda}|\Big)U_{z_1,\lambda}\frac{\partial U_{z_1,\lambda}}{\partial \overline{r}}dx\right|\\
\leq &\left|\int_{\mathbb{R}^{6}}\frac{\lambda^{2}}{(1+\lambda|x- z_{1}|)^{4}}\frac{\lambda^{2}}{(1+\lambda^{2}|x- z_{1}|^{2})^{2}}\frac{\lambda^{3}}{(1+\lambda |x- z_{i}|)^{5}}dx\right|
+\left|\int_{\mathbb{R}^{6}}\Big(|y|^{-4}\ast |U_{z_1,\lambda}|^{2}\Big)U_{z_i,\lambda}\frac{\partial U_{z_1,\lambda}}{\partial \overline{r}}dy\right|\\
\leq &\int_{\mathbb{R}^{6}}\frac{1}{(1+|x-\lambda z_{1}|)^{8}}\frac{\lambda}{(1+ |x-\lambda z_{i}|)^{5}}dx
+\frac{C}{\lambda^{4}|z_{1}-z_{i}|^{5}}\\
\leq &\frac{C}{\lambda^{4}|z_{1}-z_{i}|^{5}}.
\endaligned$$
Similarly, if $j=i$, we have
$$
\left|\int_{\mathbb{R}^{6}}\Big(|x|^{-4}\ast |U_{z_1,\lambda}U_{z_i,\lambda}|\Big)U_{z_j,\lambda}\frac{\partial U_{z_j,\lambda}}{\partial \overline{r}}dx\right|\leq \frac{C}{\lambda^{4}|z_{1}-z_{i}|^{5}}.
$$
When $j\neq1$ and $j\neq i$, we have
$$\aligned
&\left|\int_{\mathbb{R}^{6}}\Big(|x|^{-4}\ast |U_{z_1,\lambda}U_{z_i,\lambda}|\Big)U_{z_j,\lambda}\frac{\partial U_{z_j,\lambda}}{\partial \overline{r}}dx\right|\\
\leq &\frac{C}{\lambda^{4}|z_{1}-z_{i}|^{4}}\int_{\mathbb{R}^{6}}(\frac{1}{(1+|x-\lambda z_{1}|)^{4}}+\frac{1}{(1+|x-\lambda z_{i}|)^{4}})\frac{1}{(1+|x-\lambda z_{j}|)^{4}}dx\\
\leq &\frac{C}{\lambda^{4}|z_{1}-z_{i}|^{4}}\frac{1}{\lambda|z_{1}-z_{j}|}\int_{\mathbb{R}^{6}}(\frac{1}{(1+|x-\lambda z_{1}|)^{7}}+\frac{1}{(1+|x-\lambda z_{j}|)^{7}})dx\\
&+\frac{C}{\lambda^{4}|z_{1}-z_{i}|^{4}}\frac{1}{\lambda|z_{i}-z_{j}|}\int_{\mathbb{R}^{6}}(\frac{1}{(1+|x-\lambda z_{i}|)^{7}}+\frac{1}{(1+|x-\lambda z_{j}|)^{7}})dx\\
=&\frac{C}{\lambda^{4}|z_{1}-z_{i}|^{4}}\frac{1}{\lambda|z_{1}-z_{j}|}+\frac{C}{\lambda^{4}|z_{1}-z_{i}|^{4}}\frac{1}{\lambda|z_{i}-z_{j}|}.
\endaligned$$
So,
\begin{equation}\label{5.6}
\aligned
&\left|\sum_{j=1}^{m}\int_{\mathbb{R}^{6}}\Big(|x|^{-4}\ast |U_{z_1,\lambda}\sum_{i=2}^{m}U_{z_i,\lambda}|\Big)U_{z_j,\lambda}\frac{\partial U_{z_j,\lambda}}{\partial \overline{r}}dx\right|\\
\leq&\sum_{i=2}^{m}\frac{C}{\lambda^{4}|z_{1}-z_{i}|^{5}}+
\sum_{i=2}^{m}\sum_{j=2,\neq i}^{m}\frac{C}{\lambda^{4}|z_{1}-z_{i}|^{4}}\frac{1}{\lambda|z_{1}-z_{j}|}
+\sum_{i=2}^{m}\sum_{j=2,\neq i}^{m}\frac{C}{\lambda^{4}|z_{1}-z_{i}|^{4}}\frac{1}{\lambda|z_{i}-z_{j}|}\\
=&\frac{Cm^{5}}{\lambda^{4}}+
\sum_{i=2}^{m}\frac{C}{\lambda^{4}|z_{1}-z_{i}|^{4}}\frac{m}{\lambda}\\
=&\frac{C}{\lambda^{\frac{3}{2}}}+
\frac{Cm}{\lambda}\frac{m^{4}}{\lambda^{4}}=O(\frac{1}{\lambda^{1+\varepsilon}}).
\endaligned\end{equation}

Similarly, we have
\begin{equation}\label{5.9}
\int_{\mathbb{R}^{6}}\Big(|x|^{-4}\ast |U_{z_1,\lambda}\sum_{i=2}^{m}U_{z_i,\lambda}|\Big)U_{z_1,\lambda}\sum_{j=2}^{m}\frac{\partial U_{z_j,\lambda}}{\partial \overline{r}}dx=O(\frac{1}{\lambda^{1+\varepsilon}})
\end{equation}
and
\begin{equation}\label{5.10}
(m-1)\sum_{j=1,\neq2}^{m}\int_{\mathbb{R}^{6}}\Big(|x|^{-4}\ast |U_{z_1,\lambda}\sum_{i=2}^{m}U_{z_i,\lambda}|\Big)U_{z_2,\lambda}\frac{\partial U_{z_j,\lambda}}{\partial \overline{r}}dx=O(\frac{1}{\lambda^{1+\varepsilon}}).
\end{equation}
Now, recall that \eqref{5.5}, \eqref{5.6}, \eqref{5.9} and \eqref{5.10}, we have
$$
\int_{\mathbb{R}^{6}}\Big(|x|^{-4}\ast |U_{z_1,\lambda}\sum_{i=2}^{m}U_{z_i,\lambda}|\Big)Z_{\overline{r},\overline{x}'',\lambda}^{\ast}\frac{\partial Z_{\overline{r},\overline{x}'',\lambda}^{\ast}}{\partial \overline{r}}dx=O(\frac{1}{\lambda^{1+\varepsilon}}).
$$
Combining this and \eqref{5.4}, we can obtain
$$\aligned
\int_{\mathbb{R}^{6}}&\Big(|x|^{-4}\ast |Z_{\overline{r},\overline{x}'',\lambda}^{\ast}|^{2}\Big)Z_{\overline{r},\overline{x}'',\lambda}^{\ast}
	\frac{\partial Z_{\overline{r},\overline{x}'',\lambda}^{\ast}}{\partial \overline{r}}dx
	-\sum_{j=1}^{m}\int_{\mathbb{R}^{6}}\Big(|x|^{-4}\ast |U_{z_j,\lambda}|^{2}\Big)U_{z_j,\lambda}\frac{\partial Z_{\overline{r},\overline{x}'',\lambda}^{\ast}}{\partial \overline{r}}dx\\
	&=m\Big(-\sum_{j=2}^{m}
	\frac{B_{2}}{\overline{r}\lambda^{4}|z_{1}-z_{j}|^{4}}+O(\frac{1}{\lambda^{1+\varepsilon}})\Big).
	\endaligned$$
So, by Lemma \ref{P3}, we obtain
	$$
	\frac{\partial J(Z_{\overline{r},\overline{x}'',\lambda})}{\partial \overline{r}}=m\Big(\frac{B_{1}}{\lambda^{2}}\frac{\partial V(\overline{r},\overline{x}'')}{\partial \overline{r}}+\sum_{j=2}^{m}
	\frac{B_{2}}{\overline{r}\lambda^{4}|z_{1}-z_{j}|^{4}}+O(\frac{1}{\lambda^{1+\varepsilon}})\Big),
	$$
	where $B_j$, $j=1, 2$ are some positive constants.
\end{proof}

\begin{lem}\label{D3}
We have
$$\aligned
&\int_{\mathbb{R}^{6}}(-\Delta u_{m}+ V(|x'|,x'')u_{m}
-\Big(|x|^{-4}\ast |u_{m}|^{2}\Big)u_{m})\frac{\partial Z_{\overline{r},\overline{x}'',\lambda}}{\partial \overline{r}} dx\\
=&m\Big(
-\frac{B_1}{\lambda^{3}}V(\overline{r},\overline{x}'')
+\frac{m^{4}B_3}{\lambda^{5}}+O(\frac{1}{\lambda^{1+\varepsilon}})\Big).
\endaligned$$
\end{lem}
\begin{proof}
Direct calculations show
$$\aligned
\int_{\mathbb{R}^{6}}&(-\Delta u_{m}+ V(|x'|,x'')u_{m}
-\Big(|x|^{-4}\ast |u_{m}|^{2}\Big)u_{m})\frac{\partial Z_{\overline{r},\overline{x}'',\lambda}}{\partial \overline{r}} dx\\
=&\langle J'(Z_{\overline{r},\overline{x}'',\lambda}),\frac{\partial Z_{\overline{r},\overline{x}'',\lambda}}{\partial \overline{r}}\rangle+m\Big\langle -\Delta \phi+ V(r,x'')\phi
-\Big(|x|^{-4}\ast |Z_{\overline{r},\overline{x}'',\lambda}|^{2}\Big)\phi\\
&-2\Big(|x|^{-4}\ast Z_{\overline{r},\overline{x}'',\lambda}\phi\Big)Z_{\overline{r},\overline{x}'',\lambda},\frac{\partial Z_{z_1,\lambda}}{\partial \overline{r}}\Big\rangle-\int_{\mathbb{R}^{6}}\Big(|x|^{-4}\ast |u_{m}|^{2}\Big)u_{m}\frac{\partial Z_{\overline{r},\overline{x}'',\lambda}}{\partial \overline{r}} dx\\
&+\int_{\mathbb{R}^{6}}\Big(|x|^{-4}\ast |Z_{\overline{r},\overline{x}'',\lambda}|^{2}\Big)\phi\frac{\partial Z_{\overline{r},\overline{x}'',\lambda}}{\partial \overline{r}} dx+\int_{\mathbb{R}^{6}}2\Big(|x|^{-4}\ast Z_{\overline{r},\overline{x}'',\lambda}\phi\Big)Z_{\overline{r},\overline{x}'',\lambda}\frac{\partial Z_{\overline{r},\overline{x}'',\lambda}}{\partial \overline{r}} dx\\
&+\int_{\mathbb{R}^{6}}\Big(|x|^{-4}\ast |Z_{\overline{r},\overline{x}'',\lambda}|^{2}\Big)Z_{\overline{r},\overline{x}'',\lambda}\frac{\partial Z_{\overline{r},\overline{x}'',\lambda}}{\partial \overline{r}} dx\\
:=&\langle J'(Z_{\overline{r},\overline{x}'',\lambda}),\frac{\partial Z_{\overline{r},\overline{x}'',\lambda}}{\partial \overline{r}}\rangle+mI_{1}-I_{2}.
\endaligned$$

Using \eqref{c5} and \eqref{c6}, we obtain
$$
I_{1}=O(\frac{\|\phi\|_{\ast}}{\lambda^{\varepsilon}})
=O(\frac{1}{\lambda^{1+\varepsilon}}).
$$

By \eqref{c02}, we have
$$
||x|^{-4}\ast (Z_{\overline{r},\overline{x}'',\lambda}
\phi)|\leq C\|\phi\|_{\ast}\sum_{j=1}^{m}\frac{\lambda^{2}}{(1+\lambda|x-z_{j}|)^{4}}.
$$
So,
$$\aligned
&\left|\int_{\mathbb{R}^6}\Big(|x|^{-4}\ast (Z_{\overline{r},\overline{x}'',\lambda}
\phi)\Big)Z_{\overline{r},\overline{y}'',\lambda} \frac{\partial Z_{z_1,\lambda}}{\partial \overline{r}}dx\right|\\
\leq &C\|\phi\|_{\ast}\|Z_{\overline{r},\overline{x}'',\lambda}\|_{\ast}
\left|\int_{\mathbb{R}^6}\sum_{j=1}^{m}\frac{\lambda^{2}}{(1+\lambda|x-z_{j}|)^{4}}\sum_{j=1}^{m}
\frac{\lambda^{2}}{(1+\lambda|x-z_{j}|)^{2+\tau}}\xi\frac{\partial U_{z_1,\lambda}}{\partial \overline{r}}dx\right|\\
\leq &C\|\phi\|_{\ast}
\left|\int_{\mathbb{R}^6}\sum_{j=1}^{m}
\frac{\lambda^{4}}{(1+\lambda|x-z_{j}|)^{6+\tau}}\xi\frac{\partial U_{z_1,\lambda}}{\partial \overline{r}}dx\right|.
\endaligned$$
Using the same arguments Lemma \ref{D2}, it is easy to show
$$
\Big\langle\Big(|x|^{-4}\ast (Z_{\overline{r},\overline{x}'',\lambda}
\phi)\Big)Z_{\overline{r},\overline{y}'',\lambda}, \frac{\partial Z_{z_1,\lambda}}{\partial \overline{r}}\Big\rangle= O(\frac{1}{\lambda^{1+\varepsilon}}),
$$
and
$$
\Big\langle\Big(|x|^{-4}\ast |Z_{\overline{r},\overline{x}'',\lambda}|^{2}\Big)\phi, \frac{\partial Z_{z_1,\lambda}}{\partial \overline{r}}\Big\rangle= O(\frac{1}{\lambda^{1+\varepsilon}}).
$$

By direct calculations, we can prove
$$\aligned
I_{2}\leq&\int_{\mathbb{R}^{6}}\Big(|x|^{-4}\ast |\phi|^{2}\Big)|\phi||\frac{\partial Z_{\overline{r},\overline{x}'',\lambda}}{\partial \overline{r}}| dx
+\int_{\mathbb{R}^{6}}\Big(|x|^{-4}\ast |\phi|^{2}\Big)|Z_{\overline{r},\overline{x}'',\lambda}||\frac{\partial Z_{\overline{r},\overline{x}'',\lambda}}{\partial \overline{r}}| dx\\
&+2\int_{\mathbb{R}^{6}}\Big(|x|^{-4}\ast |Z_{\overline{r},\overline{x}'',\lambda}\phi|\Big)|\phi||\frac{\partial Z_{\overline{r},\overline{x}'',\lambda}}{\partial \overline{r}}| dx\\
\leq&C\int_{\mathbb{R}^{6}}\Big(|x|^{-4}\ast |\phi|^{2}\Big)|\phi|\xi\sum_{j=1}^{m}|\frac{\partial U_{z_j,\lambda}}{\partial \overline{r}}| dx
+C\int_{\mathbb{R}^{6}}\Big(|x|^{-4}\ast |\phi|^{2}\Big)|Z_{\overline{r},\overline{x}'',\lambda}|\xi\sum_{j=1}^{m}|\frac{\partial U_{z_j,\lambda}}{\partial \overline{r}}| dx\\
&+C\int_{\mathbb{R}^{6}}\Big(|x|^{-4}\ast |Z_{\overline{r},\overline{x}'',\lambda}\phi|\Big)|\phi|\xi\sum_{j=1}^{m}|\frac{\partial U_{z_j,\lambda}}{\partial \overline{r}}|dx\\
\leq&\frac{C\|\phi\|_{\ast}^{3}}{\lambda}\int_{\mathbb{R}^{6}}\Big(|x|^{-4}\ast \Big(\sum_{j=1}^{m}
\frac{\lambda^{2}}{(1+\lambda|x-z_{j}|)^{2+\tau}}\Big)^{2}\Big)\sum_{j=1}^{m}
\frac{\lambda^{2}}{(1+\lambda|x-z_{j}|)^{2+\tau}}\sum_{j=1}^{m}U_{z_j,\lambda} dx\\ &+\frac{C\|\phi\|_{\ast}^{2}\|Z_{\overline{r},\overline{x}'',\lambda}\|_{\ast}}{\lambda}\int_{\mathbb{R}^{6}}\Big(|x|^{-4}\ast \Big(\sum_{j=1}^{m}
\frac{\lambda^{2}}{(1+\lambda|x-z_{j}|)^{2+\tau}}\Big)^{2}\Big)\sum_{j=1}^{m}
\frac{\lambda^{2}}{(1+\lambda|x-z_{j}|)^{2+\tau}}\sum_{j=1}^{m}U_{z_j,\lambda} dx\\ &+\frac{C\|\phi\|_{\ast}^{2}\|Z_{\overline{r},\overline{x}'',\lambda}\|_{\ast}}{\lambda}\int_{\mathbb{R}^{6}}\Big(|x|^{-4}\ast \Big(\sum_{j=1}^{m}
\frac{\lambda^{2}}{(1+\lambda|x-z_{j}|)^{2+\tau}}\Big)^{2}\Big)\sum_{j=1}^{m}
\frac{\lambda^{2}}{(1+\lambda|x-z_{j}|)^{2+\tau}}\sum_{j=1}^{m}U_{z_j,\lambda} dx.
\endaligned$$
Using the same arguments Lemma \ref{D2} and \eqref{c15}, it is easy to show
$$
I_{2}=O(\frac{m}{\lambda^{1+\varepsilon}}).
$$

So, we have proved
$$
\Big\langle J'(Z_{\overline{r},\overline{x}'',\lambda}+\phi),\frac{\partial Z_{\overline{r},\overline{x}'',\lambda}}{\partial \overline{r}}\Big\rangle
=\Big\langle J'(Z_{\overline{r},\overline{x}'',\lambda}),\frac{\partial Z_{\overline{r},\overline{x}'',\lambda}}{\partial \overline{r}}\Big\rangle+O(\frac{m}{\lambda^{1+\varepsilon}}).
$$
From Lemma \ref{D2}, we obtain the result.
\end{proof}

Using the same arguments in Lemma \ref{D2} and Lemma \ref{D3}, we can also prove

\begin{lem}
We have
\begin{equation}\aligned\label{dr0}
	\Big\langle J'(Z_{\overline{r},\overline{x}'',\lambda}+\phi),\frac{\partial Z_{\overline{r},\overline{x}'',\lambda}}{\partial \lambda}\Big\rangle
	=&m\Big(-\frac{B_1}{\lambda^{3}}V(\overline{r},\overline{x}'')+\sum_{j=2}^{m}
	\frac{B_{2}}{\lambda^{5}|z_{1}-z_{j}|^{4}}+O(\frac{1}{\lambda^{3+\varepsilon}})\Big)\\
	=&m\Big(-\frac{B_1}{\lambda^{3}}V(\overline{r},\overline{x}'')+
	\frac{B_{3}m^{4}}{\lambda^{5}}+O(\frac{1}{\lambda^{3+\varepsilon}})\Big),
	\endaligned\end{equation}
and
\begin{equation}\label{dx0}
	\Big\langle J'(Z_{\overline{r},\overline{x}'',\lambda}+\phi),\frac{\partial Z_{\overline{r},\overline{x}'',\lambda}}{\partial \overline{x}_{j}''}\Big\rangle
	=m\Big(\frac{B_1}{\lambda^{2}}\frac{\partial V(\overline{r},\overline{x}'')}{\partial \overline{x}_{j}''}+O(\frac{1}{\lambda^{1+\varepsilon}})\Big), \ \ j=3,\cdot\cdot\cdot,6,
\end{equation}
where $B_i$, $i=1, 2,3$ are some positive constants.
\end{lem}

\begin{lem}\label{D4}
It holds
\begin{equation}\label{d20}
	\int_{\mathbb{R}^{6}}|\nabla \phi|^{2}dx+\int_{\mathbb{R}^{6}}V(r,x'') \phi^{2}dx=O(\frac{m}{\lambda^{2+\varepsilon}}).
\end{equation}
\end{lem}
\begin{proof}
It follows from \eqref{c14} that
$$
\aligned
&\int_{\mathbb{R}^{6}}|\nabla \phi|^{2}dx+\int_{\mathbb{R}^{6}}V(r,x'') \phi^{2}dx\\
=&\int_{\mathbb{R}^{6}}(\Delta Z_{\overline{r},\overline{x}'',\lambda}- V(r,x'')Z_{\overline{r},\overline{x}'',\lambda}
+\Big(|x|^{-4}\ast |Z_{\overline{r},\overline{x}'',\lambda}+
\phi|^{2}\Big)(Z_{\overline{r},\overline{x}'',\lambda}+
\phi))\phi dx\\
=&\int_{\mathbb{R}^{6}}(\Delta Z_{\overline{r},\overline{x}'',\lambda}+\Big(|x|^{-4}\ast |Z_{\overline{r},\overline{x}'',\lambda}|^{2}\Big)
Z_{\overline{r},\overline{x}'',\lambda})\phi dx- \int_{\mathbb{R}^{6}}V(r,x'')Z_{\overline{r},\overline{x}'',\lambda}\phi dx\\
&+\int_{\mathbb{R}^{6}}\Big(\Big(|x|^{-4}\ast |Z_{\overline{r},\overline{x}'',\lambda}+
\phi|^{2}\Big)(Z_{\overline{r},\overline{x}'',\lambda}+
\phi)-\Big(|x|^{-4}\ast |Z_{\overline{r},\overline{x}'',\lambda}|^{2}\Big)
Z_{\overline{r},\overline{x}'',\lambda}\Big)\phi dx\\
:=&I_{1}-I_{2}+I_{3}.
\endaligned$$

By the estimates of $\Phi_{3}$ and $\Phi_{4}$ in Lemma \ref{C5}, we obtain
$$
\left|\int_{\mathbb{R}^{6}}\Delta Z_{\overline{r},\overline{x}'',\lambda} \phi dx\right|
\leq C\frac{\|\phi\|_{\ast}}{\lambda^{1+\varepsilon}}\int_{\mathbb{R}^{6}}
\sum_{j=1}^{m}\frac{\lambda^{4}}{(1+\lambda|x-z_{j}|)^{4+\tau}}\sum_{j=1}^{m}
\frac{\lambda^{2}}{(1+\lambda|x-z_{j}|)^{2+\tau}} dx
\leq \frac{Cm}{\lambda^{2+2\varepsilon}}.
$$
Moreover, from the estimates of $\Phi_{1}$ in Lemma \ref{C5}, we obtain
$$\aligned
&\left|\int_{\mathbb{R}^{6}}\Big(|x|^{-4}\ast |Z_{\overline{r},\overline{x}'',\lambda}|^{2}\Big)
Z_{\overline{r},\overline{x}'',\lambda} \phi dx\right|\\
\leq &C\frac{\|\phi\|_{\ast}}{\lambda^{1+\varepsilon}}\int_{\mathbb{R}^{6}}
\sum_{j=1}^{m}\Big(\frac{\lambda^{4}}{(1+\lambda|x-z_{j}|)^{4+\tau}}
+\frac{\lambda^{4}}{(1+\lambda|x-z_{j}|)^{8}}\Big)\sum_{j=1}^{m}
\frac{\lambda^{2}}{(1+\lambda|x-z_{j}|)^{2+\tau}} dx
\leq \frac{Cm}{\lambda^{2+2\varepsilon}}.
\endaligned$$
So, we can prove
$$
|I_{1}|\leq  \frac{Cm}{\lambda^{2+2\varepsilon}},
$$
By (3.37) in \cite{PWY}, we have
$$
|I_{2}|\leq \frac{Cm}{\lambda^{2+2\varepsilon}}.
$$
By direct calculations, we can prove
$$
|I_{3}|
\leq C (\|\phi\|_{\ast}^{4}+\|\phi\|_{\ast}^{2})
\int_{\mathbb{R}^{6}}\Big(|x|^{-4}\ast \Big(\sum_{j=1}^{m}
\frac{\lambda^{2}}{(1+\lambda|x-z_{j}|)^{2+\tau}}\Big)^{2}\Big)\Big(\sum_{j=1}^{m}
\frac{\lambda^{2}}{(1+\lambda|x-z_{j}|)^{2+\tau}}\Big)^{2} dx
\leq \frac{Cm}{\lambda^{2+2\varepsilon}}.
$$
So, we deduce
$$
\int_{\mathbb{R}^{6}}|\nabla \phi|^{2}dx+\int_{\mathbb{R}^{6}}V(r,x'') \phi^{2}dx=O(\frac{m}{\lambda^{2+\varepsilon}}).
$$
\end{proof}

\noindent
{\bf Proof of Theorem \ref{EXS1}.}
It is easy to see that \eqref{d1} is equivalent to
\begin{equation}\label{d12}
\begin{split}
	-2\int_{D_{\rho}}|\nabla u_{m}|^{2}dx-&\frac{1}{2}\int_{D_{\rho}}(6V(x)+\langle x,\nabla V(x)\rangle) u_{m}^{2}dx
	+3\int_{D_{\rho}}\int_{\R^6}\frac{|u_{m}(y)|^{2}|u_{m}(x)|^{2}}{|x-y|^4}dxdy\\&-2
	\int_{D_{\rho}}\int_{\R^6}x(x-y)\frac{|u_{m}(y)|^{2}|u_{m}(x)|^{2}}{|x-y|^6}dxdy
	\\&=O\Big(\int_{\partial D_{\rho}}\Big(|\nabla \phi|^{2}+ \phi^{2}\Big)ds
	+\int_{\partial D_{\rho}}\Big(\int_{\R^6} \frac{|\phi(y)|^{2}}{|x-y|^4}dy\Big)|\phi|^{2}ds\Big).
\end{split}
\end{equation}

Since
$
\sum_{j=1}^{m}\int_{\mathbb{R}^{6}}[\Big(|x|^{-4}\ast |Z_{z_j,\lambda}|^{2}\Big)Z_{j,l}\phi +2\Big(|x|^{-4}\ast |Z_{z_j,\lambda}Z_{j,l}|\Big)Z_{z_j,\lambda}\phi] dx=0,
$
we obtain from \eqref{c14} that
\begin{equation}\label{d13}
\aligned
\int_{D_{\rho}}|\nabla u_{m}|^{2}dx+\int_{D_{\rho}}V(x) u_{m}^{2}dx&=\int_{D_{\rho}}\Big(|x|^{-4}\ast |u_{m}|^{2}\Big)|u_{m}|^{2} dx+\sum_{l=1}^{6}c_{l}\sum_{j=1}^{m}\int_{\mathbb{R}^{6}}\big[\Big(|x|^{-4}\ast |Z_{z_j,\lambda}|^{2}\Big)Z_{j,l}\\
&+2\Big(|x|^{-4}\ast (Z_{z_j,\lambda}
Z_{j,l})\Big)Z_{z_j,\lambda}\big]Z_{\overline{r},\overline{x}'',\lambda}dx+O\Big(\int_{\partial D_{\rho}}(|\nabla \phi|^{2}+ \phi^{2}) ds\Big).
\endaligned\end{equation}

Inserting \eqref{d13} into \eqref{d12}, we obtain
\begin{equation}\label{d14}
\aligned
\int_{D_{\rho}}(V(x)+\frac{1}{2}\langle x,\nabla V(x)\rangle) u_{m}^{2}dx=&-2\sum_{l=1}^{6}c_{l}\sum_{j=1}^{m}\int_{\mathbb{R}^{6}}\Big[\Big(|x|^{-4}\ast |Z_{z_j,\lambda}|^{2}\Big)Z_{j,l}+2\Big(|x|^{-4}\ast Z_{z_j,\lambda}Z_{j,l}\Big)Z_{z_j,\lambda}\Big]Z_{\overline{r},\overline{x}'',\lambda}dx\\
&+O\Big(\int_{\partial D_{\rho}}\Big(|\nabla \phi|^{2}+ \phi^{2}\Big)ds
+\int_{\partial D_{\rho}}\Big(\int_{\R^6}\frac{|\phi(y)|^{2}}{|x-y|^4}dy\Big)|\phi|^{2}ds\Big)\\
&+O\Big(\int_{ D_{\rho}}\int_{\R^6\backslash D_{\rho}}(x_i-y_i)\frac{|\phi(x)|^{2}|\phi(y)|^{2}}{|x-y|^6}dxdy\Big)+O(\frac{1}{\lambda^{2+\varepsilon}}), \ i=3,\cdot\cdot\cdot,6.
\endaligned\end{equation}

On the other hand, by direct calculations, we can prove
$$
\sum_{j=1}^{m}\int_{\mathbb{R}^{6}}\Big[\Big(|x|^{-4}\ast |Z_{z_j,\lambda}|^{2}\Big)Z_{j,l}+2\Big(|x|^{-4}\ast (Z_{z_j,\lambda}
Z_{j,l})\Big)Z_{z_j,\lambda}\Big]\frac{\partial Z_{\overline{r},\overline{x}'',\lambda}}{\partial \overline{r}} dx=O(m\lambda^{2}), \ \ l=2,\cdot\cdot\cdot,6,
$$
and
$$
\sum_{j=1}^{m}\int_{\mathbb{R}^{6}}\Big[\Big(|x|^{-4}\ast |Z_{z_j,\lambda}|^{2}\Big)Z_{j,1}+2\Big(|x|^{-4}\ast (Z_{z_j,\lambda}
Z_{j,1})\Big)Z_{z_j,\lambda}\Big]\frac{\partial Z_{\overline{r},\overline{x}'',\lambda}}{\partial \overline{r}} dx=O(m).
$$
Combining these and Lemma \ref{D3}, \eqref{ci}, \eqref{dr0}, \eqref{dx0} and \eqref{c16}, we can get the following estimate for $c_l$:
\begin{equation}\label{d15}
c_i=O(\frac{1}{\lambda^{3+\varepsilon}}),  i=2,\cdot\cdot\cdot,6,
\end{equation}
and
\begin{equation}\label{d16}
c_1=O(\frac{1}{\lambda^{1+\varepsilon}}).
\end{equation}
It is easy to see that
$$
\int_{\mathbb{R}^{6}}\Big[\Big(|x|^{-4}\ast |Z_{z_j,\lambda}|^{2}\Big)Z_{j,2}Z_{z_j,\lambda}dx
\leq C\int_{\mathbb{R}^{6}}\int_{\mathbb{R}^{6}}\frac{\lambda^{4}}{(1+\lambda^{2}|x-z_{j} |^{2})^{4}}\frac{1}{|x-y|^{4}}\frac{\lambda^{6}|y-z_{j} |}
{(1+\lambda^{2}|y-z_{j} |^{2})^{5}}dxdy
=O(\lambda),
$$
where $j=1,\cdot\cdot\cdot,m,$ and if $i\neq j$,
$$\aligned
\int_{\mathbb{R}^{6}}&\Big[\Big(|x|^{-4}\ast |Z_{z_j,\lambda}|^{2}\Big)Z_{j,2}Z_{z_i,\lambda}dx\\
&\leq C\int_{\mathbb{R}^{6}}\int_{\mathbb{R}^{6}}\frac{\lambda^{4}}{(1+\lambda^{2}|x-z_{j} |^{2})^{4}}\frac{1}{|x-y|^{4}}\frac{\lambda^{2}}{(1+\lambda^{2}|y-z_{i} |^{2})^{2}}\frac{\lambda^{4}|y-z_{j} |}
{(1+\lambda^{2}|y-z_{j}|^{2})^{3}}dxdy\\
&\leq C\int_{\mathbb{R}^{6}}\frac{1}{(1+|y-\lambda z_{i} |)^{4}}\frac{\lambda}
{(1+|y-\lambda z_{j}|)^{9}}dy\\
&\leq\frac{ C}{\lambda^{3}|z_{i}-z_{j}|^{4}}.
\endaligned$$
by Lemma \ref{P1}. So, for $l=2$, we have
$$
\sum_{j=1}^{m}\int_{\mathbb{R}^{6}}\Big(|x|^{-4}\ast |Z_{z_j,\lambda}|^{2}\Big)Z_{j,2}Z_{\overline{r},\overline{x}'',\lambda}dx
\leq O(m\lambda)+m\sum_{j=2}^{m}\frac{ C}{\lambda^{3}|z_{1}-z_{j}|^{4}}
=O(m\lambda).
$$
Similarly, we have,
$$\aligned
\sum_{j=1}^{m}\int_{\mathbb{R}^{6}}\Big(|x|^{-4}\ast Z_{z_j,\lambda}Z_{j,2}\Big)Z_{z_j,\lambda}Z_{\overline{r},\overline{x}'',\lambda}dx
=O(m\lambda).
\endaligned$$
Consequently,
\begin{equation}\label{eq1}
\sum_{j=1}^{m}\int_{\mathbb{R}^{6}}\Big[\Big(|x|^{-4}\ast |Z_{z_j,\lambda}|^{2}\Big)Z_{j,2}+2\Big(|x|^{-4}\ast Z_{z_j,\lambda}Z_{j,2}\Big)Z_{z_j,\lambda}\Big]Z_{\overline{r},\overline{x}'',\lambda}dx=O(m\lambda).
\end{equation}

Using the same arguments as above, we have,
\begin{equation}\label{eq2}
\sum_{j=1}^{m}\int_{\mathbb{R}^{6}}\Big[\Big(|x|^{-4}\ast |Z_{z_j,\lambda}|^{2}\Big)Z_{j,l}+2\Big(|x|^{-4}\ast Z_{z_j,\lambda}Z_{j,l}\Big)Z_{z_j,\lambda}\Big]Z_{\overline{r},\overline{x}'',\lambda}dx=O(m\lambda), \ \ l=3,4, 5, 6.
\end{equation}
and
\begin{equation}\label{eq3}
\sum_{j=1}^{m}\int_{\mathbb{R}^{6}}\Big[\Big(|x|^{-4}\ast |Z_{z_j,\lambda}|^{2}\Big)Z_{j,l}+2\Big(|x|^{-4}\ast Z_{z_j,\lambda}Z_{j,l}\Big)Z_{z_j,\lambda}\Big]Z_{\overline{r},\overline{x}'',\lambda}dx=O(\frac{m}{\lambda}), \ \ l=1.
\end{equation}
By \eqref{eq1}-\eqref{eq3}, we find from \eqref{d15} and \eqref{d16} that \eqref{d14} is equivalent to
\begin{equation}\label{d17}
\aligned
\int_{D_{\rho}}(V(x)+\frac{1}{2}\langle x,\nabla V(x)\rangle) u_{m}^{2}dx
=&O(\frac{m}{\lambda^{2+\varepsilon}})+O\Big(\int_{\partial D_{\rho}}\Big(|\nabla \phi|^{2}+ \phi^{2}\Big)ds+\int_{\partial D_{\rho}}\Big(\int_{\R^6}\frac{|\phi(y)|^{2}}{|x-y|^4}dy\Big)|\phi|^{2}ds\Big)\\
&
+O\Big(\int_{ D_{\rho}}\int_{\R^6\backslash D_{\rho}}(x_i-y_i)\frac{|\phi(x)|^{2}|\phi(y)|^{2}}{|x-y|^6}dxdy\Big)+O(\frac{1}{\lambda^{2+\varepsilon}}),i=3,\cdot\cdot\cdot,6
\endaligned\end{equation}
for some small $\varepsilon>0$.

We integrate by parts to find that \eqref{d2} is equivalent to
\begin{equation}\label{d11}
\begin{split}
	\int_{D_{\rho}}u_{m}^{2}
	\frac{\partial V(|x'|,x'')}{\partial x_{i}} dx&=O\Big(\int_{\partial D_{\rho}}\Big(|\nabla \phi|^{2}+ \phi^{2}\Big)ds
	+\int_{\partial D_{\rho}}\Big(\int_{\R^6}\frac{|\phi(y)|^{2}}{|x-y|^4}dy\Big)|\phi|^{2}ds\Big)\\&~~~+O\Big(\int_{ D_{\rho}}\int_{\R^6\backslash D_{\rho}}(x_i-y_i)\frac{|\phi(x)|^{2}|\phi(y)|^{2}}{|x-y|^6}dxdy\Big)+O(\frac{1}{\lambda^{2+\varepsilon}}), i=3,\cdot\cdot\cdot,6.
\end{split}
\end{equation}
From \eqref{d11}, we can rewrite \eqref{d17} as
\begin{equation}\label{d18}
\begin{split}
	\int_{D_{\rho}}(V(x)+\frac{r}{2}\frac{\partial V(r,x'')}{\partial r}) u_{m}^{2}dx
	&=o(\frac{m}{\lambda^{2}})+O\Big(\int_{\partial D_{\rho}}\Big(|\nabla \phi|^{2}+ \phi^{2}\Big)ds
	+\int_{\partial D_{\rho}}\Big(\int_{D_{\rho}} \frac{|\phi(y)|^{2}}{|x-y|^4}\Big)|\phi|^{2}ds\Big)\\&~~~+
	O\Big(\int_{ D_{\rho}}\int_{\R^6\backslash D_{\rho}}(x_i-y_i)\frac{|\phi(x)|^{2}|\phi(y)|^{2}}{|x-y|^6}dxdy\Big), \ i=3,\cdot\cdot\cdot,6,
\end{split}
\end{equation}
that is,
\begin{equation}\label{d19}
\begin{split}
	\int_{D_{\rho}}\frac{1}{2 r}\frac{\partial( r^{2}V(r,x''))}{\partial r} u_{m}^{2}dx
	&=o(\frac{m}{\lambda^{2}})+O\Big(\int_{\partial D_{\rho}}\Big(|\nabla \phi|^{2}+ \phi^{2}\Big)ds
	+\int_{\partial D_{\rho}}\Big(\int_{D_{\rho}} \frac{|\phi(y)|^{2}}{|x-y|^4}\Big)|\phi|^{2}ds\Big)\\&~~~+
	O\Big(\int_{ D_{\rho}}\int_{\R^6\backslash D_{\rho}}(x_i-y_i)\frac{|\phi(x)|^{2}|\phi(y)|^{2}}{|x-y|^6}dxdy\Big), \ i=3,\cdot\cdot\cdot,6.
\end{split}
\end{equation}

From Lemma 3.5 in\cite{PWY}, we know
$$
\int_{D_{4\delta}\backslash D_{3\delta}}|\nabla\phi|^{2}dx=O(\frac{m}{\lambda^{2+\varepsilon}})).
$$
Which together with Lemma \ref{D4} implies that
$$
\int_{D_{4\delta}\backslash D_{3\delta}}
\Big(|\nabla \phi|^{2}+ \phi^{2}\Big)dx
+\int_{D_{4\delta}\backslash D_{3\delta}}\Big(\int_{D_{\rho}} \frac{|\phi(y)|^{2}}{|x-y|^4}dy\Big)|\phi|^{2}dx
+\int_{ D_{4\delta}\backslash D_{3\delta}}\int_{\R^6\backslash D_{\rho}}(x_i-y_i)\frac{|\phi(x)|^{2}|\phi(y)|^{2}}{|x-y|^6}dxdy
=O(\frac{m}{\lambda^{2+\varepsilon}}),
$$
where $i=3,\cdot\cdot\cdot,6$. As a result, we can find a $\rho\in(3\delta,4\delta)$, such that
$$
\int_{\partial D_{\rho}}\Big(|\nabla \phi|^{2}+ \phi^{2}\Big)ds
+\int_{\partial D_{\rho}}\Big(\int_{D_{\rho}} \frac{|\phi(y)|^{2}}{|x-y|^4}dy\Big)|\phi|^{2}ds+
\int_{ D_{\rho}}\int_{\R^6\backslash D_{\rho}}(x_i-y_i)\frac{|\phi(x)|^{2}|\phi(y)|^{2}}{|x-y|^6}dxdy=O(\frac{m}{\lambda^{2+\varepsilon}}),
$$
where $i=3,\cdot\cdot\cdot,6$. By Lemma 3.4 in \cite{PWY}, for any $C^1$ function $g(r, x'')$, it holds
	$$
	\int_{D_{\rho}}g(r,x'') u_{m}^{2}dx=m(\frac{1}{\lambda^{2}}g(\overline{r},\overline{x}'')\int_{\mathbb{R}^{6}}U_{0,1}^{2}dx+o(\frac{1}{\lambda^{2}})).
	$$
 We can obtain from \eqref{d11} and \eqref{d19} that
$$
m(\frac{1}{\lambda^{2}}\frac{\partial V(\overline{r},\overline{x}'')}{\partial \overline{x}_{i}}\int_{\mathbb{R}^{6}}U_{0,1}^{2}dx+o(\frac{1}{\lambda^{2}}))=o(\frac{m}{\lambda^{2}}),
$$
and
$$
m(\frac{1}{\lambda^{2}}\frac{1}{2\overline{r}}\frac{\partial (\overline{r}^{2}V(\overline{r},\overline{x}''))}{\partial \overline{r}}\int_{\mathbb{R}^{6}}U_{0,1}^{2}dx+o(\frac{1}{\lambda^{2}}))=o(\frac{m}{\lambda^{2}}).
$$
Therefore, the equations to determine $(\overline{r},\overline{x}'')$ are
\begin{equation}\label{d21}
\frac{\partial V(\overline{r},\overline{x}'')}{\partial \overline{x}_{i}}=o(1),  i=3,\cdot\cdot\cdot,6,
\end{equation}
and
\begin{equation}\label{d22}
\frac{\partial (\overline{r}^{2}V(\overline{r},\overline{x}''))}{\partial \overline{r}}=o(1).
\end{equation}

We have proved that \eqref{d1}, \eqref{d2} and \eqref{d3} are equivalent to \eqref{d21}, \eqref{d22} and
$$
-\frac{B_1}{\lambda^{3}}V(\overline{r},\overline{x}'')+\frac{m^{4}B_3}{\lambda^{5}}=O(\frac{1}{\lambda^{3+\varepsilon}}).
$$
Let $\lambda=tm^{2}$, then $t\in[L_0, L_1]$ since $\lambda\in[L_0m^{2}, L_1m^{2}]$. Then, we can get
\begin{equation}\label{d23}
-\frac{B_1}{t^{3}}V(\overline{r},\overline{x}'')+\frac{B_3}{t^{5}}=o(1), t\in[L_0, L_1].
\end{equation}
Let
$$
F(t, \overline{r},\overline{x}'')=(\nabla_{\overline{r},\overline{x}''}(\overline{r}^{2}V(\overline{r},\overline{x}'')),
-\frac{2B_1}{t^{3}}V(\overline{r},\overline{x}'')+\frac{B_3}{t^{5}}).
$$
Then
$$
\mbox{deg}(F(t, \overline{r},\overline{x}''), [L_0, L_1]\times B_{\theta}((r_{0},x_{0}'')))
=-\mbox{deg}(\nabla_{\overline{r},\overline{x}''}(\overline{r}^{2}V(\overline{r},x_{0}'')),B_{\theta}((r_{0},x_{0}'')))\neq0.
$$
So, \eqref{d21}, \eqref{d22} and \eqref{d23} have a solution $t_{m}\in[L_0, L_1]$, $(\overline{r}_{m},\overline{x}_{m}'')\in B_{\theta}((r_{0},x_{0}''))$.
$\hfill{} \Box$

\vspace{1cm}

{\bf Statements.} There is no conflict of interest between the authors.and all data generated or analysed during this study are included in this published article.

\end{document}